\documentclass[a4paper]{amsart}
\usepackage{amsmath,amssymb,amsthm}
\usepackage[margin=2.6cm]{geometry}
\usepackage{hyperref}
\usepackage{xcolor}
\usepackage{mathrsfs}
\usepackage{nameref}
\usepackage{zref-xr}
\zxrsetup{toltxlabel}
\zexternaldocument*[unitary:]{unitary}

\numberwithin{equation}{section}
\allowdisplaybreaks

\newtheorem{thm}{Theorem}[section]
\newtheorem{lem}[thm]{Lemma}
\newtheorem{prop}[thm]{Proposition}
\newtheorem{cor}[thm]{Corollary}

\theoremstyle{remark}

\theoremstyle{definition}

\newcommand{\GL}{\mathrm{GL}}
\newcommand{\Q}{\mathbb{Q}}
\newcommand{\sm}{\mathrm{sm}}

\newcommand{\cts}{\mathrm{cts}}
\newcommand{\Z}{\mathbb{Z}}

\DeclareMathOperator{\supp}{supp}
\DeclareMathOperator{\val}{val}
\DeclareMathOperator{\diag}{diag}
\DeclareMathOperator{\Ind}{Ind}
\DeclareMathOperator{\Lie}{Lie}
\DeclareMathOperator{\Res}{Res}
\DeclareMathOperator{\Hom}{Hom}
\DeclareMathOperator{\Ad}{Ad}
\newcommand{\alggrp}[1]{\underline{#1}}

\renewcommand{\simeq}{\cong}

\newcommand{\congto}{\xrightarrow{\,\sim\,}}
\newcommand{\into}{\hookrightarrow}
\newcommand{\onto}{\twoheadrightarrow}

\title[On the irreducibility of $p$-adic Banach principal series of $\GL_{3}$]{On the irreducibility of $p$-adic Banach principal series of $p$-adic $\GL_3$}
\author{Noriyuki Abe}
\address[N. Abe]{Graduate School of Mathematical Sciences, the University of Tokyo, 3-8-1 Komaba, Meguro-ku, Tokyo 153-8914, Japan.}
\thanks{The first-named author was supported by JSPS KAKENHI Grant Number 18H01107.}
\email{abenori@ms.u-tokyo.ac.jp}
\author{Florian Herzig}
\address[F. Herzig]{Department of Mathematics, University of Toronto, 40 St.\ George Street, Toronto, ON M5S 2E4, Canada.}
\thanks{The second-named author was partially supported by an NSERC grant.}
\email{herzig@math.toronto.edu}

\dedicatory{Dedicated to Pham Huu Tiep on the occasion of his 60th birthday}

\begin{document}
\begin{abstract}
  We establish an optimal (topological) irreducibility criterion for $p$-adic Banach principal series of $\GL_{n}(F)$, where $F/\Q_p$ is finite and $n \le 3$.
  This is new for $n = 3$ as well as for $n = 2$, $F \ne \Q_p$ and establishes a refined version of Schneider's conjecture \cite[Conjecture 2.5]{MR2275644} for these groups.
\end{abstract}

\maketitle

\section{Introduction}
\label{sec:introduction}

Suppose that $F/\Q_p$ is a finite extension with normalized absolute value $\lvert\cdot\rvert_F$ and residue field of cardinality $q$.
This paper concerns the continuous representations of $G = \GL_n(F)$ on $p$-adic Banach spaces over a coefficient field $C$ that is a finite extension of $\Q_p$. Such Banach representations were introduced in the work of Schneider--Teitelbaum \cite{MR1900706} and play a fundamental role in the $p$-adic Langlands program (see for example \cite{MR2097893}, \cite{MR2359853}, \cite{Colmez}, \cite{emerton-local-global}, \cite{MR3150248}, \cite{MR3529394}).
Little has been known about Banach representations outside the group $\GL_2(\Q_p)$ so far.
The main goal of this paper is to determine an optimal (topological) irreducibility criterion for Banach principal series of $\GL_n(F)$ when $n \le 3$.
This goes further than Schneider's conjecture \cite[Conjecture 2.5]{MR2275644} for these groups.

Let $B$ denote the upper-triangular Borel subgroup, $T$ the diagonal maximal torus.
If $\chi = \chi_1 \otimes \cdots \otimes \chi_n : T \to C^\times$ is a continuous character, then we inflate $\chi$ to $B$ and form the parabolic induction
\begin{equation*}
  (\Ind_B^G \chi)^\cts := \{ \text{$f\colon G\to C$ continuous} \mid \text{$f(gb) = \chi(b)^{-1}f(g)$ for any $g\in G$, $b\in B$} \},
\end{equation*}
which carries a natural Banach topology making it into an (admissible) Banach representation of $G$ under left translation that we call a \emph{Banach principal series}.
If $\chi$ is smooth and we replace continuous functions by locally constant functions, then we obtain a dense smooth subrepresentation $(\Ind_B^G \chi)^\sm$ of $(\Ind_B^G \chi)^\cts$.

To state our main result, we say that a character $\lambda \colon F^\times \to C^\times$ is \emph{non-positive algebraic} if it is of the form $\lambda(t) = \prod_{\kappa\colon F\to \overline C} \kappa(t)^{a_\kappa}$
for some $(a_\kappa) \in \Z_{\le 0}^{\Hom(F,\overline C)}$, where $\overline C$ denotes an algebraic closure of $C$. 
\begin{thm}\label{thm:main-intro}
  Suppose that $n \le 3$.
  Then the Banach principal series $(\Ind_{B}^{G}\chi)^{\cts}$ is reducible if and only if there exists $1 \le i < n$ such that $\chi_{i}\chi_{i+1}^{-1}$ is non-positive algebraic.
\end{thm}

This result was known for $n = 1$, as well as for $n = 2$ and $F = \Q_p$ \cite[Proposition 2.6]{MR2275644}.
It was also known when $d\chi_{i,\kappa}-d\chi_{j,\kappa}-(j-i) \not\in \Z_{< 0}$ for all $1 \le i < j \le n$ and all $\kappa \colon F \to C$ \cite[Proposition 2.6]{MR2275644}, \cite{OS}, where $d\chi_i \colon F \otimes_{\Q_p} C \cong \bigoplus_\kappa C \to C$ denotes the derivative of $\chi_i$ (noting that $\chi_i$ is $\Q_p$-locally analytic), by locally analytic representation theory and comparison with BGG category $\mathcal O$.
It was known when $\lvert \chi_i\chi_j^{-1}(\varpi_F)\rvert < 1$ for all $1 \le i < j \le n$ \cite{MR3537231}, where $\lvert\cdot\rvert$ denotes a defining absolute value on $C$, by using continuous distribution algebras.
Finally it was known when $\chi$ is unitary and $\overline\chi_i \ne \overline\chi_{i+1}$ for all $1 \le i < n$, where $\overline\chi_i$ denotes the reduction of $\chi_i$ modulo the maximal ideal of $\mathcal O_C$, since the reduction of a $G$-stable unit ball in $(\Ind_{B}^{G}\chi)^{\cts}$ is irreducible in this case \cite{MR2255820}.
Note that the conditions in the last three results exclude the exceptional cases where $\chi_{i}\chi_{i+1}^{-1}$ is non-positive algebraic.
(Also note that some of these results assume that $\Hom(F,C) = \Hom(F,\overline C)$ or at least $\Hom(F,C)\ne \varnothing$.)

As a corollary to Theorem~\ref{thm:main-intro} we obtain a slightly improved version of \cite[Theorem \ref{unitary:thm:GLn-irred}]{unitary} for $\GL_n(F)$ ($n \ge 1$), see Corollary~\ref{cor:GLn}.
It is natural to wonder if Theorem~\ref{thm:main-intro} continues to hold for all $n$, but the evidence is rather limited at this point.

We now indicate how we prove Theorem~\ref{thm:main-intro}.
The ``if'' direction is clear by transitivity of parabolic induction, because if $n = 2$ and $\chi_1 \chi_2^{-1}$ is non-positive algebraic, then $(\Ind_{B}^{G}\chi)^{\cts}$ contains (up to twist by $\chi_1 \circ \det$) a finite-dimensional algebraic subrepresentation.

To prove the ``only if'' direction we apply our results from \cite{unitary} -- relying on locally analytic vectors \cite{ST} and the work of Orlik--Strauch \cite{OS2}, \cite{OS3} -- to reduce to the case where $\chi$ is smooth and moreover the smooth representation $(\Ind_{B}^{G}\chi)^{\sm}$ contains a non-generic irreducible subrepresentation.
Note that we may extend $C$ if necessary to prove Theorem~\ref{thm:main-intro}, hence we can apply the results of \cite{unitary}.
Our assumption then becomes that $\chi_i \ne \chi_{i+1}$ for all $1 \le i < n$, and it remains to consider the case where $n = 3$ and $\chi_1\chi_3^{-1} = \lvert\cdot\rvert_F$ (because Bernstein--Zelevinsky proved that every irreducible subrepresentation is generic otherwise).
We remark that generically $(\Ind_B^G \chi)^\sm$ is indecomposable of length 2 in this case.
By symmetry, using the outer automorphism of $\GL_3$, we reduce to proving the following proposition.

\begin{prop}\label{prop:final case-intro}
Assume that $\chi$ is smooth and that $\chi_{1}\ne \chi_{2}$, $\chi_{2}\ne \chi_{3}$.
If $\chi_{1}\chi_{3}^{-1} \ne \lvert \cdot\rvert_F^2$ and $\lvert \chi_{2}^{-1}\chi_{3}(\varpi_{F})\rvert > \lvert q\rvert$, then $(\Ind_{B}^{G}\chi)^{\cts}$ is irreducible.
\end{prop}

By \cite[Corollary~\ref{unitary:cor:irreduciblity}]{unitary} it suffices to show that any irreducible subrepresentation $\pi$ of $(\Ind_{B}^{G}\chi)^{\sm}$ is dense in $(\Ind_{B}^{G}\chi)^{\cts}$.
Let $P$ be the parabolic corresponding to $3 = 2+1$, with unipotent radical $N$ and Levi subgroup $L$ containing $T$.
Let $Z_L$ denote the center of $L$.
Let $N_0 := N \cap \GL_3(\mathcal O_F)$, $L^+ := \{ z \in Z_L \mid z N_0 z^{-1} \subset N_0\}$, and $Z_L^+ := Z_L \cap L^+$.
Then $\pi^{N_0}$ carries a Hecke action of $L^+$ by letting $\ell \in L^+$ act as $[N_{0}:\ell N_{0}\ell^{-1}]^{-1}\sum_{n\in N_{0}/\ell N_{0}\ell^{-1}}n\ell v$ for $v\in \pi^{N_{0}}$,
and $\pi^{N_0,Z_L^+=\chi}$ becomes a smooth representation of $L$, with $L^+$ acting via the Hecke action and $Z_L$ via $\chi$.
In fact, the natural map $\pi^{N_0,Z_L^+ = \chi} \to \pi_N^{Z_L = \chi}$ is an $L$-linear isomorphism \cite[Proposition 4.3.4]{MR2292633}, where $\pi_N$ denotes the unnormalized Jacquet module.
We use the geometric lemma~\cite[5.2~Theorem]{MR0579172} to show that
\begin{equation*}
  0 \ne \pi^{N_0,Z_L^+ = \chi} \into ((\Ind_{B}^{G}\chi)^{\sm})^{N_{0},Z_{L}^{+} = \chi} \xrightarrow[\theta]{\;\sim\;} (\Ind_{B\cap L}^{L}\chi)^{\sm},
\end{equation*}
where the final isomorphism $\theta$ is induced by restriction of functions from $G$ to $L$.

Our key step is to find an explicit inverse of the isomorphism $\theta$, see Theorem~\ref{thm:explicit}.
As $\chi_1 \ne \chi_2$, the $L$-subrepresentation $\pi^{N_0,Z_L^+ = \chi}$ has to either be the full principal series or a twist of the Steinberg representation, and we can apply $\theta^{-1}$ to obtain many explicit functions in $\pi$.
We now put $\sigma := (\Ind_{B\cap L}^{L}\chi)^{\cts}$ and think of $(\Ind_{B}^{G}\chi)^{\cts}$ as $(\Ind_{P}^{G}\sigma)^{\cts}$.
Then we are able to find sequences $h_n \in \pi$ and $h_n' \in (\Ind_{P}^{G}\sigma)^{\cts}$ ($n \ge 1$) and a vector $v \in \sigma$ such that 
\begin{itemize}
\item $\supp(h_n') \subset \overline N_0 P$, where $\overline N_0$ is the transpose of $N_0$;
\item $h'_{n}(x)\in Cv$ for all $x \in \overline N_0$;
\item the sequence $h_n'$ is bounded away from 0;
\item $\lim_{n\to\infty}(h_{n} - h'_{n}) = 0$.
\end{itemize}
Using Corollary~\ref{cor:irreducibility criterion}, which may be of independent interest for establishing the irreducibility of continuous parabolic inductions, we deduce that $\pi$ is dense in $(\Ind_{P}^{G}\sigma)^{\cts}$.

\subsection{Notation}
\label{sec:notation}
Let $C$ be a finite extension of $\Q_{p}$, $\mathcal{O}_{C}$ the ring of integers of $C$, and $\overline C$ a choice of algebraic closure.
We fix a uniformizer $\varpi_{C}$ of $C$ and an absolute value $\lvert\cdot\rvert$ on $C$.
In this paper, unless otherwise stated, the coefficient field of any representation is $C$.
Let $\alggrp{G}$ be a connected reductive group over $\Q_{p}$, $\alggrp{Z}_{\alggrp{G}}$ the center of $\alggrp{G}$, $G = \alggrp{G}(\Q_{p})$ is the group of rational points.
We use the same notation for other groups.
For a parabolic subgroup $\alggrp{P}$ of $\alggrp{G}$, Levi subgroup $\alggrp{L}$, unipotent radical $\alggrp{N}$ and a representation $\sigma$ of $L$, we consider the following two inductions.
1) If $\sigma$ is a Banach representation, namely a continuous representation on a $p$-adic Banach space, then $(\Ind_{P}^{G}\sigma)^{\cts}$ is the space of continuous maps $f\colon G\to \sigma$ such that $f(g\ell n) = \sigma(\ell)^{-1}f(g)$ for $g\in G$, $\ell\in L$ and $n\in N$, and $G$ acts via $(\gamma f)(g) = f(\gamma^{-1}g)$ for $g, \gamma \in G$.
Let $K$ be a compact open subgroup of $G$ such that $KP = G$ and $P\cap K = (L\cap K)(N\cap K)$.
We can always choose a defining norm $\lvert\cdot\rvert$ on $\sigma$ that is $L\cap K$-invariant.
Then we define the $K$-invariant norm $\lVert\cdot\rVert$ on $(\Ind_{P}^{G}\sigma)^{\cts}$ by $\lVert f\rVert = \sup_{x\in K}\lvert f(x)\rvert$, and $(\Ind_{P}^{G}\sigma)^{\cts}$ is a Banach representation with this norm.
2) If $\sigma$ is smooth, then $(\Ind_{P}^{G}\sigma)^{\sm}$ is the space of locally constant functions $G\to \sigma$ such that $f(g\ell n) = \sigma(\ell)^{-1}f(g)$ for $g\in G$, $\ell\in L$ and $n\in N$, and $G$ acts via $(\gamma f)(g) = f(\gamma^{-1}g)$ for $g, \gamma \in G$.

We say that a continuous representation of a topological group is \emph{irreducible} if it is topologically irreducible.

\subsection{Acknowledgement}
We thank the referees for helpful comments.
In particular, we thank the referee who asked rationality questions which led us to remove the assumption that the coefficient field $C$ be sufficiently large.
Part of this work was done during a pleasant stay of the first-named author at University of Toronto.

\section{Preliminaries}

\subsection{Irreducibility criterion}
\label{subsec:Criterion}
We recall an irreducibility criterion from \cite{unitary}.
We assume that the derived subgroup of $\alggrp{G}$ is simply connected.
Let $\alggrp{P} = \alggrp{L}\alggrp{N}$ be a parabolic subgroup and $\sigma$ an absolutely irreducible Banach representation of $L$.
We assume that $\sigma$ is finite-dimensional.
Then $\sigma$ is locally analytic (as the ground field is $\Q_p$) and by \cite[Lemma~\ref{unitary:lem:sigma0-tau-decomp}]{unitary}, after perhaps replacing $C$ by a finite extension, there exist locally analytic representations $\sigma_{0},\tau$ of $L$ and a parabolic subgroup $\alggrp{Q} = \alggrp{L}_{\alggrp Q}\alggrp{N}_{\alggrp Q}$ containing $\alggrp P$ such that
\begin{itemize}
\item $\sigma \simeq \sigma_{0}\otimes\tau$.
\item $\sigma_{0}$ is simple as $\Lie(L)\otimes_{\Q_{p}}C$-module and $\tau$ is smooth.
\item Let $L(\sigma_{0}')$ be the simple $\Lie(G)\otimes_{\Q_{p}}C$-module in the BGG category $\mathcal{O}$ such that $L(\sigma_{0}')^{\Lie(N)}\simeq \sigma_{0}'$, where $\sigma_{0}'$ is the dual of $\sigma_{0}$.
Then $L(\sigma_{0}')$ is locally $\Lie(Q)\otimes_{\Q_{p}}C$-finite and $\alggrp{Q}$ is maximal subject to this condition.
\item $L(\sigma_{0}')$ has the structure of a locally finite $Q$-locally analytic representation such that $q X q^{-1} = \Ad(q)(X)$ on $L(\sigma_{0}')$ for all $q \in Q$, $X \in \Lie(G)$ and whose restriction to $P$ on $L(\sigma_{0}')^{\Lie(N)}\simeq \sigma_{0}'$ is the given one.
(This structure is unique if it exists, cf.\ \cite[Section~\ref{unitary:sec:funct-orlik-stra}]{unitary}.)
\end{itemize}
The decomposition $\sigma \cong \sigma_{0}\otimes\tau$ is moreover unique up to smooth characters of $L_Q$ \cite[Lemma~\ref{unitary:lem:semisimple-lift}(ii)]{unitary}.

Then we have the following.

\begin{thm}[{\cite[Corollary~\ref{unitary:cor:irreduciblity}, Theorem~\ref{unitary:thm:equivalence-irred}]{unitary}}]\label{thm:criterion}
Assume $p > 2$ (resp.\ $p > 3$) if the absolute root system of $\alggrp{G}$ has irreducible components of type $B$, $C$ or $F_{4}$
(resp.\ $G_{2}$).
The following are equivalent.
\begin{enumerate}
\item $(\Ind_{P}^{G}\sigma)^{\cts}$ is irreducible;
\item $(\Ind_{P \cap L_{Q}}^{L_{Q}}\tau)^{\cts}$ is irreducible;
\item any irreducible subrepresentation of $(\Ind_{P\cap L_{Q}}^{L_{Q}}\tau)^{\sm}$ is dense in $(\Ind_{P\cap L_{Q}}^{L_{Q}}\tau)^{\cts}$.
\end{enumerate}
\end{thm}

\subsection{Density lemmas}
\label{sec:density}
To prove that a certain subrepresentation is dense, we will use the following.
This is a generalization of \cite[Lemma~\ref{unitary:lem:supported on big cell}]{unitary}.
Recall that $K$ denotes a compact open subgroup of $G$ such that $KP = G$ and $P\cap K = (L\cap K)(N\cap K)$.

\begin{lem}\label{lem:irred1}
Let $\alggrp{P} = \alggrp{L}\alggrp{N}$ be a parabolic subgroup, $\overline{\alggrp{P}} =\alggrp{L}\overline{\alggrp{N}}$ the opposite parabolic subgroup, and assume that $P\cap K = (L\cap K)(N\cap K)$.
Let $\sigma$ be an irreducible Banach representation of $L$ with central character and an $L\cap K$-stable unit ball $\sigma^0$. Let $\pi\subset (\Ind_{P}^{G}\sigma)^\cts$ be a closed subrepresentation and $v\in \sigma$. We assume the following: there exists a compact open subgroup $\overline{N}_{0}\subset \overline{N}\cap K$ such that for any $k\in\Z_{>0}$, there exists $f \in \pi$
satisfying the following conditions:
\begin{enumerate}
\item we have $f(K) \subset \sigma^0$, $f(K) \not\subset \varpi_C \sigma^0$;
\item for any $x\in K\setminus \overline{N}_{0}(P\cap K)$ we have $f(x)\in \varpi_{C}^{k}\sigma^{0}$;
\item for any $n\in \overline{N}_{0}$ we have $f(n)\in Cv + \varpi_{C}^{k}\sigma^{0}$.
\end{enumerate}
Then we have $\pi = (\Ind_{P}^{G}\sigma)^\cts$.
\end{lem}
\begin{proof}
We choose an $L\cap K$-invariant norm $\lvert\cdot\rvert$ on $\sigma$ with unit ball $\sigma^0$ and $\lvert\sigma\rvert = \lvert C\rvert$.
By the assumption we know that $v \ne 0$ (otherwise $f(K) \subset \varpi_{C}^{k}\sigma^{0}$), and for convenience we scale $v$ such that $|v| = 1$.
We get a $K$-invariant norm $\lVert\cdot\rVert$ on $(\Ind_{P}^{G}\sigma)^{\cts}$, as above, and we let $((\Ind_{P}^{G}\sigma)^{\cts})^{0}$ be its unit ball.
Then condition (i) above just means that $\lVert f\rVert = 1$.
We introduce the following notation.
For any open and closed subset $X\subset G/P$, set $V(X) := \{f\in (\Ind_{P}^{G}\sigma)^{\cts}\mid \supp(f)\subset X\}$.
This is a closed subspace of $(\Ind_{P}^{G}\sigma)^{\cts}$ and we have $(\Ind_{P}^{G}\sigma)^{\cts} = V(X)\oplus V((G/P)\setminus X)$ as Banach spaces. 
(Namely, $\lVert f_{1} + f_{2}\rVert = \max(\lVert f_{1}\rVert,\lVert f_{2}\rVert)$ for $f_{1}\in V(X)$ and $f_{2}\in V((G/P)\setminus X)$).
We set $\lVert f\rVert_{X} := \lVert f|_{X}\rVert$.

We say that $f \in \pi$ satisfies $\mathcal{P}(\overline{N}_{0},f,k)$ if $f(K) \subset \sigma^0$ and conditions (ii), (iii) in the lemma hold.

First we prove that for any $k\in \Z_{>0}$ and for any compact open subgroup $\overline{N}'_{0}\subset \overline{N}_{0}$ there exists $f\in \pi$ with $\lVert f\rVert = 1$
such that $\mathcal{P}(\overline{N}'_{0},f,k)$ holds.
There exists $z\in Z_{L}$ such that $z\overline{N}_{0}z^{-1}\subset \overline{N}'_{0}$.
We set $X := \overline{N}_{0}P/P\subset G/P$ and $Y := (G/P)\setminus X$.
The element $z$ induces a topological isomorphism $(\Ind_{P}^{G}\sigma)^{\cts}\congto (\Ind_{P}^{G}\sigma)^{\cts}$ which sends $V(X)$ to $V(zX)$.
Hence it induces a topological isomorphism $V(X)\congto V(zX)$ and likewise $V(Y)\congto V(zY)$.
Take $r_{1}\in \Z_{\ge 0}$ (resp.\ $r_{2}\in \Z_{\le 0}$) such that $\lVert zf\rVert_{zX}\ge \lvert \varpi_{C}^{r_{1}}\rvert \lVert f\rVert_{X}$ (resp.\ $\lVert zf\rVert_{zY}\le \lvert \varpi_{C}^{r_{2}}\rvert\lVert f\rVert_{Y}$) for any $f\in V(X)$ (resp.\ $f\in V(Y)$).

For a given $k\in\Z_{> 0}$ (and $z \in Z_L$ as above), put $k' := k + r_{1} - r_{2} - \min(\val(\omega_{\sigma}(z)),0) \in \Z_{>0}$, where $\omega_{\sigma}$ denotes the central character of $\sigma$.
We take $f\in \pi$, $\lVert f\rVert = 1$ such that $\mathcal{P}(\overline{N}_{0},f,k')$ holds.
Then we have $\lVert f\rVert_{Y}\le \lvert\varpi_{C}^{k'}\rvert < 1$.
Since $\lVert f\rVert = \max(\lVert f\rVert_{X},\lVert f\rVert_{Y})$ we have $\lVert f\rVert_{X}= 1$.
Therefore $\lVert zf\rVert_{zX} \ge \lvert\varpi_{C}\rvert^{r_{1}}$ and $\lVert zf\rVert_{zY} \le \lvert\varpi_{C}\rvert^{k' + r_{2}}$.
Hence $\varpi_{C}^{-r_{1}}zf$ satisfies $\lVert \varpi_{C}^{-r_{1}}zf\rVert_{X} \ge 1$ and $\lVert \varpi_{C}^{-r_{1}}zf\rVert_{zY} \le \lvert\varpi_{C}\rvert^{k - \min(\val(\omega_{\sigma}(z)),0)}\le \lvert \varpi_{C}\rvert^{k}$.
Taking $r\in\Z_{\ge0}$ such that $f' = \varpi_{C}^{r - r_{1}}zf$ has norm $\lVert f'\rVert = 1$, 
we have $\lVert f'\rVert_{zY} \le \lvert\varpi_{C}\rvert^{k + r}\le \lvert\varpi_{C}\rvert^{k}$.
If $x\in K\setminus \overline{N}'_{0}(P\cap K)$, then the image of $x$ in $G/P$ is not in $\overline{N}'_{0}P/P$.
Hence $z^{-1}xP/P = z^{-1}xzP/P$ is not in $\overline{N}_{0}P/P$ since $\overline{N}_{0}\subset z^{-1}\overline{N}'_{0}z$.
Therefore $z^{-1}xP/P\notin X = (G/P)\setminus Y$, hence $xP/P\in zY$.
Hence $\lvert f'(x)\rvert\le \lVert f'\rVert_{zY}\le \lvert\varpi_{C}^{k}\rvert$, i.e.\ $f'(x)\in \varpi_{C}^{k}\sigma^{0}$.
For $n\in \overline{N}'_{0}$, if $n\in zY$, then $f'(n)\in \varpi_{C}^{k}\sigma^{0}$, as we have proved.
If $n\in zX$, then $z^{-1}nz\in \overline{N}_{0}$.
Hence $f'(n) = \varpi_{C}^{r - r_{1}}f(z^{-1}n) = \varpi_{C}^{r - r_{1}}\omega_{\sigma}(z)f(z^{-1}nz)\in \varpi_{C}^{r - r_{1}}\omega_{\sigma}(z)(Cv + \varpi_{C}^{k'}\sigma^{0}) =  Cv + \varpi_{C}^{k' + r - r_{1}}\omega_{\sigma}(z)\sigma^{0}$.
We have $k' + r - r_{1} = k + r - r_{2} - \min(\val(\omega_{\sigma}(z)),0) \ge k - \val(\omega_{\sigma}(z))$, as $r \ge 0$ and $r_{2} \le 0$.
Hence $\varpi_{C}^{k' + r - r_{1}}\omega_{\sigma}(z)\in \varpi_{C}^{k}\mathcal{O}_{F}$ and $\mathcal{P}(\overline{N}'_{0},f',k)$ holds.

Now we prove the lemma.
Let $\pi(\overline{N}_{0},k)$ be the $\mathcal{O}_C$-submodule of $((\Ind_{P}^{G}\sigma)^{\cts})^{0}$ consisting all $f\in \pi$ which satisfy $\mathcal P(\overline{N}_{0},f,k)$. It is $\overline{N}_0$-stable.
Define a subspace $V(\overline{N}_{0},v)$ of $V(\overline{N}_{0}P/P)$ as the space of functions $f\in V(\overline{N}_{0}P/P)$ such that $f(n)\in Cv$ for any $n\in\overline{N}_{0}$.
This is a closed subspace of $V(\overline{N}_{0}P/P)$ and the restriction to $\overline{N}_{0}$ gives an isometric isomorphism $V(\overline{N}_{0},v) \congto C^{0}(\overline{N}_{0},Cv)$, where $C^{0}(\overline{N}_{0},Cv)$ is the space of continuous functions $f\colon \overline{N}_{0}\to Cv$ with the supremum norm.
Hence the submodule $V(\overline{N}_{0},v)^{0}$ corresponds to $C^{0}(\overline{N}_{0},Cv)^{0} = C^{0}(\overline{N}_{0},\mathcal{O}_{C}v)$. (The last equality follows from $|v| = 1$.)
We prove that for any $1\le j\le k$ and for any $h\in V(\overline{N}_{0},v)^{0}$ there exists $f\in \pi(\overline{N}_{0},k)$ such that $\lVert h - f\rVert\le \lvert\varpi_{C}^{j}\rvert$ by induction on $j$.

Let $j = 1$.
We first note that $(\Ind_{P}^{G}\sigma)^{\cts}\cong (\Ind_{P\cap K}^{K}\sigma)^{\cts}$ (as $PK = G$)
and hence
\begin{equation}
  ((\Ind_{P}^{G}\sigma)^{\cts})^{0}/\varpi_{C}((\Ind_{P}^{G}\sigma)^{\cts})^{0} \cong (\Ind_{P\cap K}^{K}\sigma^{0}/\varpi_{C}\sigma^{0})^{\sm}\label{eq:red-parab}
\end{equation}
as $K$-representations.
We introduce several subspaces of $(\Ind_{P\cap K}^{K}\sigma^{0}/\varpi_{C}\sigma^{0})^{\sm}$.
For an open and closed subset $X\subset G/P\simeq K/(P\cap K)$, we put $\overline{V}(X) := \{f\in (\Ind_{P\cap K}^{K}\sigma^{0}/\varpi_{C}\sigma^{0})^{\sm}\mid \supp(f)\subset X\}$.
The restriction to $\overline{N}_{0}$ induces an isomorphism $\overline{V}(\overline{N}_{0}(P\cap K)/(P\cap K))\congto C^{\infty}(\overline{N}_{0},\sigma^{0}/\varpi_{C}\sigma^{0})$ as $\overline{N}_{0}$-representations, where $C^{\infty}(\overline{N}_{0},-)$ denotes the space of locally constant functions.
Let $\overline{v}\in \sigma^{0}/\varpi_{C}\sigma^{0}$ be the image of $v$.
We denote the inverse image of the space of locally constant functions $C^{\infty}(\overline{N}_{0},(\mathcal{O}_{C}/(\varpi_{C}))\overline{v})$ in $\overline{V}(\overline{N}_{0}(P\cap K)/(P\cap K))$ by $\overline{V}(\overline{N}_{0},\overline{v})$.

Let $\overline{\pi}(\overline{N}_{0},k)$ be the image of $\pi(\overline{N}_{0},k)\to (\Ind_{P\cap K}^{K}\sigma^{0}/\varpi_{C}\sigma^{0})^{\sm}$ (via the isomorphism~\eqref{eq:red-parab}).
Then $\overline{\pi}(\overline{N}_{0},k)$ is non-zero, since there exists $f\in \pi(\overline{N}_{0},k)$ with $\lVert f\rVert =1$ by assumption.
We also have $\overline{\pi}(\overline{N}_{0},k)\subset \overline{V}(\overline{N}_{0},\overline{v})$ by the definition of $\pi(\overline{N}_{0},k)$.
We prove that equality holds.
The subspace $\overline{\pi}(\overline{N}_{0},k)$ is non-zero and $\overline{N}_{0}$-stable, so it contains non-zero $\overline{N}_{0}$-fixed vectors.
The space of $\overline{N}_{0}$-fixed vectors in $\overline{V}(\overline{N}_{0},\overline{v}) \cong C^{\infty}(\overline{N}_{0},(\mathcal{O}_{C}/(\varpi_{C}))\overline{v})$ is one-dimensional and spanned by $g_{\overline{N}_{0}}$ which is defined by $g_{\overline{N}_{0}}(n) = \overline{v}$ for any $n\in\overline{N}_{0}$.
Hence $g_{\overline{N}_{0}}\in \overline{\pi}(\overline{N}_{0},k)$.
Recall that we have proved that for any compact open subgroup $\overline{N}'_{0}\subset \overline{N}_{0}$, there exists $f\in \pi$ with $\lVert f\rVert = 1$ such that $\mathcal P(\overline{N}'_{0},f,k)$ holds.
Hence by the same argument, for any compact open subgroup $\overline{N}'_{0}\subset \overline{N}_{0}$ we have $g_{\overline{N}'_{0}}\in \overline{\pi}(\overline{N}'_{0},k) \subset \overline{\pi}(\overline{N}_{0},k)$.
Since these elements generate $\overline{V}(\overline{N}_{0},\overline{v})$ as an $\overline{N}_{0}$-representation, we deduce $\overline{\pi}(\overline{N}_{0},k) = \overline{V}(\overline{N}_{0},\overline{v})$.
Now for any $h\in V(\overline{N}_{0},v)^{0}$, the composition $\overline{N}_{0}\xrightarrow{h} \mathcal{O}_{C}v\to (\mathcal{O}_{C}/(\varpi_{C}))\overline{v}$ lies in $\overline{V}(\overline{N}_{0},\overline{v})$.
Hence there exists $f\in \pi(\overline{N}_{0},k)$ such that $h - f\in \varpi_{C} ((\Ind_{P}^{G}\sigma)^{\cts})^{0}$.
Therefore $\lVert h - f\rVert \le \lvert\varpi_{C}\rvert$.

Let $j > 1$.
From the case of $j = 1$, there exists $f\in \pi(\overline{N}_{0},k)$ such that $\lVert h - f\rVert\le \lvert\varpi_{C}\rvert$.
Namely, $h - f = \varpi_{C} f'$ for some $f'\in ((\Ind_{P}^{G}\sigma)^{\cts})^{0}$.
Take $f_{1},f'_{1}\in V(\overline{N}_{0}P/P)$ and $f_{2},f'_{2}\in V((G/P)\setminus (\overline{N}_{0}P/P))$ such that $f = f_{1} + f_{2}$ and $f' = f'_{1} + f'_ {2}$.
Since $h\in V(\overline{N}_{0}P/P)$, we have $h - f_{1} = \varpi_{C}f'_{1}$ and $-f_{2} = \varpi_{C}f'_{2}$.
By the definition of $\pi(\overline{N}_{0},k)$, we have $\lVert f_{2}\rVert\le\lvert\varpi_{C}^{k}\rvert$ and therefore $\lVert f'_{2}\rVert\le\lvert\varpi_{C}^{k - 1}\rvert$.
By the Hahn--Banach theorem there is a continuous linear map $a\colon \sigma\to C$ such that $a(v) = 1$ and $\lvert a(w)\rvert \le \lvert w\rvert$ for any $w\in\sigma$.
Define $f''_{1},f'''_{1}\in V(\overline{N}_{0}P/P)$ by $f''_{1}(n) = a(f'_{1}(n))v$ for $n\in\overline{N}_{0}$ and $f'''_{1} = f'_{1} - f''_{1}$.
By the definition of $\pi(\overline{N}_{0},k)$, we have $f(n)\in Cv + \varpi_{C}^{k}\sigma^{0}$ for any $n\in\overline{N}_{0}$ and we also have $h(n)\in Cv$ for any $n\in\overline{N}_{0}$.
Hence $f'_{1}(n) = \varpi_{C}^{-1}(h(n) - f_{1}(n))\in Cv + \varpi_{C}^{k - 1}\sigma^{0}$.
For $n\in \overline{N}_{0}$ take $c\in C$ such that $f'_{1}(n) - cv\in \varpi_{C}^{k - 1}\sigma^{0}$.
Then $\lvert f''_{1}(n) - cv\rvert = \lvert a(f'_{1}(n) - cv)v\rvert\le \lvert f'_{1}(n) - cv \rvert\le \lvert\varpi_{C}^{k - 1}\rvert$.
Hence $\lvert f'''_{1}(n) \rvert\le \max(\lvert f'_{1}(n) - cv\rvert,\lvert f''_{1}(n) - cv\rvert)\le \lvert\varpi_{C}^{k - 1}\rvert$ for any $n\in\overline{N}_{0}$.
Therefore $\lVert f'''_{1}\rVert\le\lvert\varpi_{C}^{k - 1}\rvert$, since $f'''_{1}\in V(\overline{N}_{0}P/P)$.
By definition, $f''_{1}\in V(\overline{N}_{0},v)$.
We also have that for $\lVert f''_{1}\rVert = \lVert f'_{1} - f'''_{1}\rVert \le \max(\lVert f'\rVert,\lVert f'''_{1}\rVert)\le 1$.
Hence $f''_{1}\in V(\overline{N}_{0},v)^{0}$.
Therefore by the inductive hypothesis there exists $f''\in \pi(\overline{N}_{0},k)$ such that $\lVert f''_{1} - f''\rVert\le\lvert\varpi_{C}^{j - 1}\rvert$.
Then $f + \varpi_{C}f'' \in \pi(\overline{N}_{0},k)$ and \begin{align*}
  \lVert h - (f + \varpi_{C}f'')\rVert &= \lvert\varpi_{C}\rvert \lVert f' - f''\rVert = \lvert\varpi_{C}\rvert \lVert f''_1 + f'''_1 + f'_{2} - f''\rVert \\
& = \lvert\varpi_{C}\rvert \lVert f'_{2} + (f''_{1} - f'') + f'''_{1}\rVert \le \lvert\varpi_{C}^{j}\rvert,
\end{align*}
as required.

In particular, putting $j = k$, there exists $f = f_{k}\in \pi$ such that $\lVert h - f_{k}\rVert\le\lvert\varpi_{C}^{k}\rvert$.
Hence $h = \lim_{k\to\infty}f_{k}\in \pi$.
In other words, for any continuous function $F\colon \overline{N}_{0}\to C$, $h = F\otimes v\in \pi$, where $F\otimes v$ is defined by $\overline{N}_{0}\ni n\mapsto F(n)v$ and we regard this as an element in $V(\overline{N}_{0}P/P)$ as above.
More generally, let $F$ be a continuous function $\overline{N}\to C$ such that $\supp(F)$ is compact and define $f = F\otimes v$.
We extend this to $\overline{N}P$ by $f(np) = \sigma(p)^{-1}f(n)$ and further to $G$ by $f|_{G\setminus \overline{N}P} = 0$.
We take $z\in Z_{L}$ such that $z\supp(F)z^{-1}\subset \overline{N}_{0}$.
Then $\supp(zf) \in \overline{N}_{0}P/P$ and for $n\in \overline{N}_{0}$ we have $(zf)(n) = \omega_\sigma(z) f(z^{-1} nz) \in Cv$.
Hence $zf\in \pi$ and therefore $f\in \pi$.

Let $\ell\in L$ and suppose $F$ is a continuous function $\overline{N}\to C$ such that $\supp(F)$ is compact.
We define $(\ell F)(n) = F(\ell^{-1}n\ell)$ for $n\in \overline{N}$.
Then we have $F\otimes \ell v = \ell(\ell^{-1}F\otimes v)\in \pi$.
As $\sigma$ is irreducible and by continuity, we have $F\otimes v'\in \pi$ for any $v'\in \sigma$.
(If $\supp(F) \subset \overline N \cap K$, then $F \otimes v_i \to F \otimes v$ if $v_i \to v$, as $\lVert F\otimes v\rVert = \sup_{n \in \overline N} \lvert F(n)\rvert\cdot \lvert v\rvert$ in this case.
In general, use the action of $L$ to reduce to that case.)
In particular, $C^{0}(\overline{N}_{0},C)\otimes_{C}\sigma \subset \pi$.

Now let $f\colon \overline{N}_{0}\to \sigma$ be a continuous function and define $f_{k}\colon \overline{N}_{0}\xrightarrow{f} \sigma\to \sigma/\varpi_{C}^k\sigma^{0}$ for $k \ge 1$.
This is a locally constant function, and therefore its image is a finite subset of $\sigma/\varpi_{C}^k\sigma^{0}$.
Hence there exists $f'_{k}\in C^{0}(\overline{N}_{0},C)\otimes_{C}\sigma$ such that $f(n) - f'_{k}(n)\in \varpi_{C}^{k}\sigma^{0}$ for any $n\in\overline{N}_{0}$.
Consider $f$, $f'_{k}$ as elements of $V(\overline{N}_{0}P/P)$ as above.
Then $\lVert f - f'_{k}\rVert\le\lvert\varpi_{C}^{k}\rvert$ and hence $f = \lim_{k\to\infty}f'_{k} \in\pi$.
Now $C^{0}(\overline{N}_{0},\sigma)\simeq V(\overline{N}_{0}P/P)$ generates $(\Ind_{P}^{G}\sigma)^{\cts}$ as a $G$-representation.
Hence $\pi = (\Ind_{P}^{G}\sigma)^{\cts}$.
\end{proof}

We use this lemma in the following form, where again we choose an $L\cap K$-invariant defining norm $\lvert\cdot\rvert$ on $\sigma$ and denote by $\lVert\cdot\rVert$ the induced $K$-invariant norm on $(\Ind_{P}^{G}\sigma)^{\cts}$.
\begin{cor}\label{cor:irreducibility criterion}
Let $\alggrp{P} = \alggrp{L}\alggrp{N}$ be a parabolic subgroup, $\overline{\alggrp{P}} = \alggrp{L}\overline{\alggrp{N}}$ be the opposite parabolic subgroup, and assume that $P\cap K = (L\cap K)(N\cap K)$.
Let $\overline{N}_{0}\subset \overline{N}$ be a compact open subgroup, $\sigma$ an irreducible Banach representation of $L$ having a central character, $\pi\subset (\Ind_{P}^{G}\sigma)^{\cts}$ a closed subrepresentation, and $v\in \sigma$.
Assume that we have sequences $h_{n}\in \pi$ and $h'_{n}\in (\Ind_{P}^{G}\sigma)^{\cts}$ for $n \ge 1$ such that $\supp(h'_{n})\subset \overline{N}_{0}P$, $h'_{n}(x)\in Cv$ for any $x\in \overline{N}_{0}$, $\inf_{n}\lVert h'_{n}\rVert > 0$, and $\lim_{n\to\infty}(h_{n} - h'_{n}) = 0$.
Then we have $\pi = (\Ind_{P}^{G}\sigma)^{\cts}$.
\end{cor}
\begin{proof}
By replacing $\lvert\cdot\rvert$ by an equivalent norm we may assume that $\lvert\sigma\rvert = \lvert C\rvert$.
Note that $h_n' \ne 0$ for all $n$, as $\inf_{n}\lVert h'_{n}\rVert > 0$.
Take $c_{n}\in C^{\times}$ such that $\lVert c_{n}h'_{n}\rVert = 1$.
Then $\lvert c_{n}\rvert$ is bounded above because $\inf_{n}\lVert h'_{n}\rVert > 0$.
Therefore $\lim_{n\to\infty}(c_{n}h_{n} - c_{n}h'_{n}) = 0$.
By replacing $h_{n},h'_{n}$ with $c_{n}h_{n},c_{n}h'_{n}$ respectively, we may assume that $\lVert h'_{n}\rVert = 1$ for all $n$.
Let $k\in\Z_{>0}$ and take $n$ such that $\lVert h_{n} - h'_{n}\rVert\le \lvert\varpi_{C}^{k}\rvert$.
Since $\lVert h'_{n}\rVert =1$, we also have $\lVert h_{n}\rVert =1$.
If $x\in K$, then $h_{n}(x) - h'_{n}(x)\in \varpi_{C}^{k}\sigma^{0}$. 
Therefore, if $x\in \overline{N}_{0}$ we have $h_{n}(x)\in Cv+\varpi_{C}^{k}\sigma^{0}$ (as $h'_n(x)\in Cv$)
and if $x \in K\setminus \overline{N}_{0}(P\cap K)$ we have $h_{n}(x)\in\varpi_{C}^{k}\sigma^{0}$ (as $h'_n(x)=0$).
\end{proof}

\subsection{The group \texorpdfstring{$\GL_{2}(F)$}{GL\_2(F)}}
\label{sec:group-gl2}
When $G = \GL_{2}(F)$ for some finite extension $F$ of $\Q_{p}$ we have the following irreducibility criterion.

Let $\alggrp{G} := \Res_{F/\Q_{p}}\GL_{2}$.
Let $\alggrp{B}$ be the subgroup of upper-triangular matrices and $\alggrp{T}$ the subgroup of diagonal matrices in $\alggrp{G}$.
Let $\chi\colon T\to C^{\times}$ be a character given by $\chi(\diag(t_{1},t_{2})) = \chi_{1}(t_{1})\chi_{2}(t_{2})$, where $\chi_{i}\colon F^{\times}\to C^{\times}$ is a continuous character for $i = 1,2$.
Let $\Hom(F,\overline C)$ be the set of continuous field homomorphisms $F\to \overline C$.
\begin{thm}\label{thm:GL2}
The Banach representation $(\Ind_{B}^{G}\chi)^{\cts}$ is reducible if and only if there exists $(k_{\kappa})_{\kappa}\in\Z_{\le 0}^{\Hom(F,\overline C)}$ such that $\chi_{1}\chi_{2}^{-1}(t) = \prod_{\kappa\in\Hom(F,\overline C)}\kappa(t)^{k_{\kappa}}$.
\end{thm}
When $F = \Q_{p}$, this is stated in \cite[Proposition 2.6]{MR2275644} (though we do not know a reference for the complete proof).
\begin{proof}
We first suppose that $C$ is sufficiently large, so in particular $\Hom(F,C) = \Hom(F,\overline C)$, i.e.\ $\alggrp G$ splits over $C$.
Assume that $\chi_{1}\chi_{2}^{-1}(t) = \prod_{\kappa\in\Hom(F,C)}\kappa(t)^{k_{\kappa}}$ for some $(k_{\kappa})\in\Z_{\le 0}^{\Hom(F,C)}$.
We have $(\Ind_{B}^{G}\chi)^{\cts}\simeq (\chi_{2}\circ\det)\otimes(\Ind_{B}^{G}\chi_{1}\chi_{2}^{-1}\boxtimes \mathbf{1})^{\cts}$, where $\mathbf{1}$ is the trivial character of $F^{\times}$.
The space of rational functions in $(\Ind_{B}^{G}\chi_{1}\chi_{2}^{-1}\boxtimes \mathbf{1})^{\cts}$ is an irreducible finite-dimensional (hence closed) $G$-subrepresentation.
Hence $(\Ind_{B}^{G}\chi)^{\cts}$ is reducible.
The converse follows from \cite[Theorem~\ref{unitary:thm:GLn-irred}]{unitary} (or \cite[Theorem~\ref{unitary:thm:rank-1}]{unitary}).

For general $C$ it remains to show for any finite Galois extension $C'/C$ that if $(\Ind_{B}^{G}\chi)^{\cts} \otimes_C C' = (\Ind_{B}^{G}\chi \otimes_C C')^{\cts}$ is reducible, then so is $(\Ind_{B}^{G}\chi)^{\cts}$.
We may assume that $C'$ is sufficiently large for the previous paragraph to hold.
If $(\Ind_{B}^{G}\chi \otimes_C C')^{\cts}$ is reducible, the previous paragraph then shows that it has a finite-dimensional subrepresentation that is generated by a $\overline U$-invariant vector, where $\overline{\alggrp{U}}$ is the unipotent radical of the opposite Borel subgroup.
(This is unaffected by the twist by $\chi_{2}\circ\det$.)
On the other hand, it is easy to check that the space of $\overline U$-invariants in $(\Ind_{B}^{G}\chi \otimes_C C')^{\cts}$ has dimension at most 1, hence exactly 1.
By Galois descent it is then defined over $C$, hence so is the subrepresentation it generates.
\end{proof}

\section{The group \texorpdfstring{$\GL_{3}(F)$}{GL\_3(F)}}
\label{sec:group-gl3}
Let $F$ be a finite extension of $\Q_{p}$, $\mathcal{O}_{F}$ the ring of integers, $\varpi_{F}$ a uniformizer of $F$, $q$ the cardinality of the residue field of $F$ and $\val \colon F^\times \onto \Z$ the normalized valuation of $F$.
We normalize the norm $\lvert\cdot\rvert_{F}$ on $F$ by $\lvert\varpi_{F}\rvert_{F} = q^{-1}$, namely $\lvert x\rvert_{F} = q^{-\val(x)}$.
Let $\alggrp{G} := \Res_{F/\Q_p} \GL_{3}$, $\alggrp{B}$ the subgroup of upper-triangular matrices, $\alggrp{T}$ the subgroup of diagonal matrices, and $\alggrp{U}$ the subgroup of unipotent upper-triangular matrices. Let $\chi\colon T\to C^{\times}$ be a continuous character given by $\chi(\diag(t_{1},t_{2},t_{3})) = \chi_{1}(t_{1})\chi_{2}(t_{2})\chi_{3}(t_{3})$, where $\chi_{i}\colon F^{\times}\to C^{\times}$ is a continuous character for $i = 1,2,3$.
The main theorem of this paper is the following.
\begin{thm}\label{thm:main}
The Banach representation $(\Ind_{B}^{G}\chi)^{\cts}$ is reducible if and only if there exists $(k_{\kappa})_{\kappa}\in\Z_{\le 0}^{\Hom(F,\overline C)}$ such that $\chi_{1}\chi_{2}^{-1}(t) = \prod_{\kappa\in\Hom(F,\overline C)}\kappa(t)^{k_{\kappa}}$ or $\chi_{2}\chi_{3}^{-1}(t) = \prod_{\kappa\in\Hom(F,\overline C)}\kappa(t)^{k_{\kappa}}$.
\end{thm}

We can now slightly improve \cite[Theorem~\ref{unitary:thm:GLn-irred}]{unitary}, adding the more restrictive condition (i).
We assume for this result that $C$ is sufficiently large so that \cite{unitary} applies.

\begin{cor}\label{cor:GLn}
  Let $G = \GL_{n}(F)$, $B$ the upper-triangular Borel subgroup, and $T$ the diagonal maximal torus with Lie algebra $\mathfrak{t}$.
  Let $\chi = \chi_{1}\otimes\cdots\otimes\chi_{n}\colon T = (F^\times)^n \to C^{\times}$ be a continuous (hence locally $\Q_{p}$-analytic) character.
  We have $d\chi\in \Hom_{\Q_{p}}(\mathfrak{t},C)\simeq \bigoplus_{\kappa\in\Hom(F,C)}\Hom_{C}(\mathfrak{t} \otimes_{F,\kappa} C,C)$ and let $\lambda_{\kappa} = (\lambda_{\kappa,1},\ldots,\lambda_{\kappa,n})$ be the $\kappa$-component of $d\chi$, where $\lambda_{\kappa,k}\in \Hom_{C}(C,C)\simeq C$.
  Choose $0 = n_0 < n_1 < \cdots < n_r = n$ such that $\lambda_{\kappa,i} - \lambda_{\kappa,i + 1}\in \Z_{\le 0}$ for all $\kappa\colon F\to C$ is equivalent to $i \not\in \{n_1,\dots,n_r\}$.
  Assume that there exists no $n_k < i < j \le n_{k+1}$ (for some $0 \le k < r$) such that
  \begin{enumerate}
  \item if $n_{k+1}-n_k = 3$, then $j-i = 1$, and
  \item $\chi_{i}\chi_{j}^{-1}(t) = \lvert t\rvert_F^{j - i - 1}\prod_{\kappa\colon F\to C}\kappa(t)^{\lambda_{\kappa,i} - \lambda_{\kappa,j}}$ for all $t\in F^{\times}$.
  \end{enumerate}
  Then $(\Ind_{B}^{G}\chi)^{\cts}$ is absolutely irreducible.
\end{cor}
\begin{proof}
  Let $\alggrp{G} = \Res_{F/\Q_{p}}\GL_{n}$.
  For $\sigma := \chi$, take $\sigma_{0},\tau,\alggrp{Q}$ as in subsection \ref{subsec:Criterion}.
  Then condition (ii) is equivalent to $\tau_i\tau_j^{-1} \ne \lvert \cdot\rvert_F^{j - i - 1}$ (cf.\ the proof of \cite[Theorem~\ref{unitary:thm:GLn-irred}]{unitary}),
  and $\alggrp Q$ is the standard parabolic corresponding to the partition $\{1,\dots,n_1\}$, $\{n_1+1,\dots,n_2\}$, \dots\ of $\{1,2,\dots,n\}$.
    By Theorem~\ref{thm:criterion} it suffices to show that $(\Ind_{B\cap L_Q}^{L_Q}\tau)^{\cts}$ is absolutely irreducible.
  By \cite[Proposition~\ref{unitary:prop:irreducibility for product group}]{unitary} we may reduce to the case where $r = 1$, i.e.\ $\alggrp{Q} = \alggrp{G}$.
  Then $(\Ind_{B}^{G}\tau)^{\cts}$ is irreducible by Theorem~\ref{thm:main} if $n = 3$ and by \cite[Theorem~\ref{unitary:thm:GLn-irred}]{unitary} if $n \ne 3$.
\end{proof}

We now prepare for the proof of Theorem~\ref{thm:main}.
We first prove the ``if'' part.
Without loss of generality, using the same automorphism as in the proof of Lemma~\ref{lem:implies-thm}, suppose that $\chi_{1}\chi_{2}^{-1}(t) = \prod_{\kappa\in\Hom(F,\overline C)}\kappa(t)^{k_{\kappa}}$.
Then by Theorem~\ref{thm:GL2} there exists a subrepresentation $0 \subsetneq \pi \subsetneq (\Ind_{B \cap L}^L \chi)^\cts$, where $\alggrp{P} = \alggrp{L}\alggrp{N}$ is the standard parabolic subgroup corresponding to $3 = 2 + 1$.
Hence by transitivity of parabolic induction we get the closed subrepresentation $0 \subsetneq (\Ind_P^G \pi)^\cts \subsetneq (\Ind_{B}^G \chi)^\cts$.

For the ``only if'' part, if $(\Ind_{B}^G \chi)^\cts$ is reducible, then it is reducible over any finite extension of $C$, so we may assume that $C$ is as large as we need to apply the results of \cite{unitary}.
In particular, $\Hom(F,C) = \Hom(F,\overline C)$.
Let $\sigma := \chi$ and take $\sigma_{0},\tau,\alggrp{Q}$ as in subsection \ref{subsec:Criterion}.
More concretely, $\alggrp{Q}$ is given as follows:
\begin{itemize}
\item $\alggrp{Q} = \alggrp{B}$ if $d\chi_{i}-d\chi_{i + 1} \ne \sum_{\kappa\in \Hom(F,C)}k_{\kappa}\kappa$ for all $(k_{\kappa})\in \Z_{\le 0}^{\Hom(F,C)}$ and $i = 1,2$.
\item $\alggrp{Q}$ is the standard parabolic subgroup corresponding to $3 = 2 + 1$ if there exists $(k_{\kappa})\in \Z_{\le 0}^{\Hom(F,C)}$ such that $d\chi_{i}-d\chi_{i + 1} = \sum_{\kappa\in \Hom(F,C)}k_{\kappa}\kappa$ for $i = 1$ but not for $i = 2$.
\item $\alggrp{Q}$ is the standard parabolic subgroup corresponding to $3 = 1 + 2$ if there exists $(k_{\kappa})\in \Z_{\le 0}^{\Hom(F,C)}$ such that $d\chi_{i}-d\chi_{i + 1} = \sum_{\kappa\in \Hom(F,C)}k_{\kappa}\kappa$ for $i = 2$ but not for $i = 1$.
\item $\alggrp{Q} = \alggrp{G}$ if there exist $(k_{\kappa,i})\in \Z_{\le 0}^{\Hom(F,C)}$ such that $d\chi_{i}-d\chi_{i + 1} = \sum_{\kappa\in \Hom(F,C)}k_{\kappa,i}\kappa$ for $i = 1,2$.
\end{itemize}
By Theorem~\ref{thm:criterion}, $(\Ind_{B}^{G}\chi)^{\cts}$ is irreducible if and only if $(\Ind_{B \cap L_{Q}}^{L_{Q}}\tau)^{\cts}$ is irreducible.
Hence, if $\alggrp{Q} = \alggrp{B}$, $(\Ind_{B}^{G}\chi)^{\cts}$ is irreducible and the theorem follows.
If $\alggrp{Q} \ne \alggrp{B}$ and $\alggrp{Q} \ne \alggrp{G}$, then the theorem follows from Theorem~\ref{thm:GL2}.
Therefore we may assume $\alggrp{Q} = \alggrp{G}$ and $\chi = \tau$.
Namely we may assume $\chi$ is smooth.
So our task is to prove the following.
\begin{prop}\label{prop:main thm, smooth}
Let $\chi$ be a smooth character of $T$.
The Banach representation $(\Ind_{B}^{G}\chi)^{\cts}$ is reducible if and only if $\chi_{1} = \chi_{2}$ or $\chi_{2} = \chi_{3}$.
\end{prop}

\begin{lem}\label{lem:many case}
Proposition~\ref{prop:main thm, smooth} is true if $\chi_{1}\chi_{3}^{-1} \ne \lvert \cdot\rvert_{F}$.
\end{lem}
\begin{proof}
It is sufficient to prove that if $\chi_{1}\ne \chi_{2}$, $\chi_{2}\ne \chi_{3}$ and $\chi_{1}\chi_{3}^{-1} \ne \lvert \cdot\rvert_{F}$, then $(\Ind_{B}^{G}\sigma)^{\cts}$ is irreducible.
This follows from \cite[Theorem~\ref{unitary:thm:GLn-irred}]{unitary}.
\end{proof}

In the rest of this paper we prove the following.

\begin{prop}\label{prop:final case}
Assume that $\chi$ is smooth and that $\chi_{1}\ne \chi_{2}$, $\chi_{2}\ne \chi_{3}$.
If $\chi_{1}\chi_{3}^{-1} \ne \lvert \cdot\rvert_F^2$
and $\lvert \chi_{2}^{-1}\chi_{3}(\varpi_{F})\rvert > \lvert q\rvert$, then $(\Ind_{B}^{G}\chi)^{\cts}$ is irreducible.
\end{prop}

\begin{lem}\label{lem:implies-thm}
Proposition~\ref{prop:final case} implies Proposition~\ref{prop:main thm, smooth}, hence Theorem~\ref{thm:main}.
\end{lem}
\begin{proof}
By Lemma~\ref{lem:many case}, it is sufficient to prove that if $\chi_{1}\ne\chi_{2}$, $\chi_{2}\ne \chi_{3}$ and $\chi_{1}\chi_{3}^{-1} = \lvert \cdot\rvert_{F}$, then $(\Ind_{B}^{G}\chi)^{\cts}$ is irreducible.
As $\chi_{1}\chi_{3}^{-1} \ne \lvert \cdot\rvert_F^2$, Proposition~\ref{prop:final case} implies that $(\Ind_{B}^{G}\chi)^{\cts}$ is irreducible if $\lvert \chi_{2}^{-1}\chi_{3}(\varpi_{F})\rvert > \lvert q\rvert$.

Define $\iota\colon \GL_{3}\to \GL_{3}$ by $\iota(g) = \dot w_{0}\cdot {}^{t}g^{-1}\cdot \dot w_{0}$, where $\dot w_{0}$ is a lift of the longest element of the Weyl group.
Since $\iota(B) = B$, we have $(\Ind_{B}^{G}\chi)^{\cts}\circ\iota \cong (\Ind_{B}^{G}\chi\circ\iota)^{\cts}$.
Therefore $(\Ind_{B}^{G}\chi)^{\cts}$ is irreducible if and only if $(\Ind_{B}^{G}\chi\circ\iota)^{\cts}$ is.
Hence, from the first paragraph of this proof, $(\Ind_{B}^{G}\chi)^\cts$ is irreducible if $\lvert \chi_{1}^{-1}\chi_{2}(\varpi_{F})\rvert > \lvert q\rvert$.
Since $\lvert\chi_{1}^{-1}\chi_{3}(\varpi_{F})\rvert = \lvert q\rvert$, we have $\lvert \chi_{1}^{-1}\chi_{2}(\varpi_{F})\rvert > \lvert q\rvert$ or $\lvert \chi_{2}^{-1}\chi_{3}(\varpi_{F})\rvert > \lvert q\rvert$.
Hence we get Theorem~\ref{thm:main}.
\end{proof}

\subsection{Jacquet modules}
\label{sec:jacquet-modules}
To prove Proposition \ref{prop:final case}, we use Jacquet modules in smooth representation theory.
Let $\alggrp{P} = \alggrp{L}\alggrp{N}$ be the standard parabolic subgroup corresponding to $3 = 2 + 1$, and recall that $\alggrp{Z}_{\alggrp L}$ denotes the center of $\alggrp{L}$.
Fix a compact open subgroup $N_{0}$ of $N$ and set $L^+ := \{\ell\in L\mid \ell N_{0}\ell^{-1}\subset N_{0}\}$ and $Z_{L}^{+} := Z_L \cap L^+$.
If $\pi$ is an admissible smooth representation of $G$, then $\ell\in L^{+}$ acts on the subspace of $N_{0}$-fixed vectors $\pi^{N_{0}}$ in $\pi$ by the Hecke action
\begin{equation}
\tau_{\ell}(v) := \frac{1}{[N_{0}:\ell N_{0}\ell^{-1}]}\sum_{n\in N_{0}/\ell N_{0}\ell^{-1}}n\ell v = \int_{N_0} n\ell v\, dn\label{eq:hecke}
\end{equation}
for $v\in \pi^{N_{0}}$, where the Haar measure is normalized such that the volume of $N_0$ is 1.
Set $\pi^{N_{0},Z_{L}^{+} = \chi} := \{v\in \pi^{N_{0}}\mid \text{$\tau_{z}(v) = \chi(z)v$ for all $z\in Z_{L}^{+}$}\}$, which has a natural action of $L$,
with $L^+$ acting via~\eqref{eq:hecke} and $Z_L$ acting via $\chi$. (This is well defined by \cite[Proposition 3.3.6]{MR2292633}.)

Let $\pi_{N}$ be the space of $N$-coinvariants (i.e.\ the unnormalized Jacquet module) and define $\pi_{N}^{Z_{L} = \chi}$ analogously to above.
Then by \cite[Propositions~3.4.9, 4.3.4]{MR2292633}, the natural projection $\pi^{N_{0}}\to \pi_{N}$ induces an $L$-linear isomorphism $\pi^{N_{0},Z_{L}^{+} = \chi}\simeq \pi_{N}^{Z_{L} = \chi}$.

Let $\pi := (\Ind_{B}^{G}\chi)^{\sm}$.
Then by the geometric lemma~\cite[5.2~Theorem]{MR0579172}, $\pi_{N}$ has a filtration $0 = F_{0}\subset F_{1}\subset F_{2}\subset F_{3} = (\Ind_{B}^{G}\chi)^{\sm}_{N}$ such that $F_{3}/F_{2}\simeq (\Ind_{B\cap L}^{L}\chi)^{\sm}$, $F_{2}/F_{1}\simeq (\Ind_{B\cap L}^{L}\chi')^{\sm}$ and $F_{1}/F_{0}\simeq (\Ind_{B\cap L}^{L}\chi'')^{\sm}$, where $\chi' := \chi_{1}\boxtimes (\chi_{3}\lvert\cdot\rvert_{F})\boxtimes (\chi_{2}\lvert \cdot\rvert_{F}^{-1})$ and $\chi'' := (\chi_{2}\lvert\cdot\rvert_{F})\boxtimes(\chi_{3}\lvert\cdot\rvert_{F})\boxtimes(\chi_{1}\lvert\cdot\rvert_F^{-2})$.
Hence if $\chi_{2}\ne \chi_{3}\lvert\cdot\rvert_{F}$ and $\chi_{1}\ne \chi_{3}\lvert\cdot\rvert_{F}^{2}$, then $\chi|_{Z_{L}}\ne \chi'|_{Z_{L}}$, and $\chi|_{Z_{L}}\ne \chi''|_{Z_{L}}$.
Therefore $((\Ind_{B}^{G}\chi)^{\sm})_{N}^{Z_{L} = \chi}\simeq (\Ind_{B\cap L}^{L}\chi)^{\sm}$.
Hence if $\chi_{2}\ne \chi_{3}\lvert\cdot\rvert_{F}$ and $\chi_{1}\ne \chi_{3}\lvert\cdot\rvert_{F}^{2}$ we have an $L$-linear isomorphism
\begin{equation}\label{eq:jacquet module}
((\Ind_{B}^{G}\chi)^{\sm})^{N_{0},Z_{L}^{+} = \chi}\simeq (\Ind_{B\cap L}^{L}\chi)^{\sm}.
\end{equation}
Note that the isomorphism is induced by the restriction $(\Ind_{B}^{G}\chi)^{\sm}\ni f\mapsto f|_{L}\in (\Ind_{B\cap L}^{L}\chi)^{\sm}$.
(This follows either by the proof of the geometric lemma, or because the restriction is easily seen to induce a non-zero map $((\Ind_{B}^{G}\chi)^{\sm})_{N}^{Z_{L} = \chi} \to (\Ind_{B\cap L}^{L}\chi)^{\sm}$, hence an isomorphism.)

\subsection{Explicit formulas}
\label{sec:expl-formula}
We now take $N_{0} := N\cap K$ with $K := \GL_{3}(\mathcal{O}_{F})$.
We calculate the inverse of the map \eqref{eq:jacquet module} explicitly.

Let $\eta_i := \chi_i^{-1} \chi_{i+1}$. For a character $\eta \colon F^{\times} \to C^{\times}$ let $c(\eta) \in \Z_{\ge 0}$ denote
the conductor of $\eta$, i.e.\ the smallest integer $c \ge 0$ such that $\eta$ is trivial on $(1+(\varpi_{F}^{c})) \cap \mathcal{O}_{F}^{\times}$.
We also use the following representatives of simple reflections of the Weyl group:
\begin{equation}\label{eq:reps}
  \dot s_1 := \begin{pmatrix}0 & -1 & 0\\ 1 & 0 & 0 \\ 0 & 0 & 1\end{pmatrix}, \quad  \dot s_2 := \begin{pmatrix}1 & 0 & 0\\ 0 & 0 & -1 \\ 0 & 1 & 0\end{pmatrix}.
\end{equation}
We let $\dot w_0 := \dot s_1 \dot s_2 \dot s_1 = \dot s_2 \dot s_1 \dot s_2$.
Then $\{e,\dot s_1,\dot s_2,\dot s_1\dot s_2,\dot s_2\dot s_1,\dot w_0\}$ is a full set of representatives for $S_3$.
Let $I$ be the ``upper'' Iwahori subgroup, namely $I$ is the set of $g\in \GL_{3}(\mathcal{O}_{F})$ such that $g\pmod{\varpi_{F}}$ is an upper-triangular matrix.
Then we have $G = \coprod_{w\in S_{3}}I\dot{w}B$, where $\dot{w}$ are our representatives of $w \in S_{3}$, so to specify
$f\in ((\Ind_{B}^{G}\chi)^{\sm})^{N\cap K,Z_{L}^{+} = \chi}$ it suffices to describe the values of $f$ on $(N\cap K)\backslash I \dot w B/B$ for each $w \in S_3$, and this is what we will do now.
(In fact, it will be convenient to describe it on a slightly larger set.) We normalize the Haar measure on $F$ such that the volume of $\mathcal{O}_{F}$ is $1$.
If $\mathrm{C}$ denotes any condition, then $\delta_{\mathrm{C}} = 1$ if $\mathrm{C}$ holds and $\delta_{\mathrm{C}} = 0$ otherwise.

\begin{thm}\label{thm:explicit}
Assume that $\eta_{2}\ne \lvert\cdot\rvert_{F}^{-1}$ and $\eta_{1}\eta_{2}\ne \lvert \cdot\rvert_{F}^{-2}$.
Let $f\in ((\Ind_{B}^{G}\chi)^{\sm})^{N\cap K,Z_{L}^{+} = \chi}$.
Then
\begin{align*}
f\begin{pmatrix}1 & 0 & 0\\ a & 1 & 0 \\ b & c & 1\end{pmatrix}
& =
  \delta_{\val(b) \ge c(\eta_1\eta_2)}\displaystyle\int_{\mathcal{O}_{F}} \eta_2(1+ct) f\begin{pmatrix}1 & 0 & 0\\ a + bt & 1 & 0\\ 0 & 0 & 1\end{pmatrix}dt 
    \quad (a\in \mathcal{O}_{F},b,c\in (\varpi_{F})),\\
f\left(\begin{pmatrix}1 & a & 0\\ 0 & 1 & 0 \\ c & b & 1\end{pmatrix}\dot s_{1}\right)
& =
  \delta_{\val(b) \ge c(\eta_1\eta_2)}
  \displaystyle\int_{\mathcal{O}_{F}} \eta_2(1+ct) f\left(\begin{pmatrix}1 & a + bt & 0\\ 0 & 1 & 0\\ 0 & 0 & 1\end{pmatrix}\dot s_{1}\right)dt
\quad (a\in \mathcal{O}_{F},b,c\in (\varpi_{F})),\\
f\left(\begin{pmatrix}1 & 0 & 0\\ a & 1 & 0\\ c & 0 & 1\end{pmatrix}\dot s_{2}\right)
& =
\delta_{\val(c) \ge c(\eta_1\eta_2)} \Bigg\{
\int_{\mathcal{O}_{F}} \eta_2(t) \left[f\begin{pmatrix}1 & 0 & 0\\ a+ct & 1 & 0\\ 0 & 0 & 1\end{pmatrix} - f\begin{pmatrix}1 & 0 & 0\\ a & 1 & 0\\ 0 & 0 & 1\end{pmatrix}\right] dt \\
& \quad + \delta_{c(\eta_2) = 0}\cdot \frac{q - 1}{q - \eta_{2}(\varpi_{F})}  f\begin{pmatrix}1 & 0 & 0\\ a & 1 & 0\\ 0 & 0 & 1\end{pmatrix}\Bigg\}
\quad(a\in \mathcal{O}_{F},c\in (\varpi_{F})),\\
f\left(\begin{pmatrix}1 & a & 0\\ 0 & 1 & 0\\ 0 & c & 1\end{pmatrix}\dot s_{1}\dot s_{2}\right)
& =
\delta_{\val(c) \ge c(\eta_1\eta_2)} \Bigg\{
\int_{\mathcal{O}_{F}} \eta_2(t) \left[f\left(\begin{pmatrix}1 & a-ct & 0\\ 0 & 1 & 0\\ 0 & 0 & 1\end{pmatrix}\dot s_{1}\right) - f\left(\begin{pmatrix}1 & a & 0\\ 0 & 1 & 0\\ 0 & 0 & 1\end{pmatrix}\dot s_{1}\right)\right] dt \\
& \quad + \delta_{c(\eta_2) = 0}\cdot \frac{q - 1}{q - \eta_{2}(\varpi_{F})}  f\left(\begin{pmatrix}1 & a & 0\\ 0 & 1 & 0\\ 0 & 0 & 1\end{pmatrix}\dot s_1\right)\Bigg\}
\quad(a\in \mathcal{O}_{F},c\in (\varpi_{F})),\\
f\left(\begin{pmatrix}1 & 0 & 0\\ a & 1 & 0\\ 0 & 0 & 1\end{pmatrix}\dot s_{2}\dot s_{1}\right)
& = \frac{\delta_{c(\eta_1\eta_2)=0} (q-1)}{q(q^2-\eta_1\eta_2(\varpi_{F}))}
\Bigg\{
\int_{\mathcal{O}_{F}} \eta_2(t) \left[f\begin{pmatrix}1 & 0 & 0\\ a+t & 1 & 0\\ 0 & 0 & 1\end{pmatrix} - f\begin{pmatrix}1 & 0 & 0\\ a & 1 & 0\\ 0 & 0 & 1\end{pmatrix}\right] dt \\
& \quad + \delta_{c(\eta_2) = 0}\cdot \frac{q - 1}{q - \eta_{2}(\varpi_{F})} f\begin{pmatrix}1 & 0 & 0\\ a & 1 & 0\\ 0 & 0 & 1\end{pmatrix}
+ \displaystyle\int_{(\varpi_{F})} \eta_2(1-at) f\left(\begin{pmatrix}1 & t & 0\\ 0 & 1 & 0\\ 0 & 0 & 1\end{pmatrix}\dot s_{1}\right)dt\Bigg\}\\
&\hspace{10cm} (a\in \mathcal{O}_{F}),\\
f\left(\begin{pmatrix}1 & a & 0\\ 0 & 1 & 0\\ 0 & 0 & 1\end{pmatrix}\dot w_{0}\right)
& = \frac{\delta_{c(\eta_1\eta_2)=0} (q-1)}{q(q^2-\eta_1\eta_2(\varpi_{F}))}
\Bigg\{
\int_{\mathcal{O}_{F}} \eta_2(t) \left[f\left(\begin{pmatrix}1 & a-t & 0\\ 0 & 1 & 0\\ 0 & 0 & 1\end{pmatrix}\dot s_{1}\right) - f\left(\begin{pmatrix}1 & a & 0\\ 0 & 1 & 0\\ 0 & 0 & 1\end{pmatrix}\dot s_{1}\right)\right] dt \\
& \quad + \delta_{c(\eta_2) = 0}\cdot \frac{q - 1}{q - \eta_{2}(\varpi_{F})} f\left(\begin{pmatrix}1 & a & 0\\ 0 & 1 & 0\\ 0 & 0 & 1\end{pmatrix}\dot s_{1}\right)
+ \displaystyle\int_{(\varpi_{F})} \eta_2(-1+at) f\begin{pmatrix}1 & 0 & 0\\ t & 1 & 0\\ 0 & 0 & 1\end{pmatrix} dt\Bigg\}\\
&\hspace{10cm} (a\in \mathcal{O}_{F}).
\end{align*}
\end{thm}

Here the quantity \[
\delta_{c(\eta_{2}) = 0}\frac{1}{q - \eta_{2}(\varpi_{F})}
\]
is well defined and independent of our choice of $\varpi_F$: if $c(\eta_{2}) \ne 0$, then this is zero and if $c(\eta_{2}) = 0$ then $\eta_{2}(\varpi_{F})\ne \lvert \varpi_{F}\rvert_F^{-1} = q$ from the assumption $\eta_{2}\ne \lvert\cdot\rvert_{F}^{-1}$.
Similarly,
\[
\frac{\delta_{c(\eta_1\eta_2)=0} (q-1)}{q(q^2-\eta_1\eta_2(\varpi_{F}))}
\]
is well defined and independent of our choice of $\varpi_F$ because $\eta_{1}\eta_{2}\ne \lvert \cdot\rvert_{F}^{-2}$.

We prove Theorem~\ref{thm:explicit} in this subsection.
For $k = 2,4,6$, the $k$-th formula in the theorem follows from the $(k - 1)$-th formula by replacing $f$ with $\dot s_{1}^{-1}f$.
(When $k = 6$ it helps to observe that $\eta_2(-1) = \eta_1(-1)$ if the formula is nonzero, as this only happens when $\eta_1\eta_2$ is unramified.)

Moreover, we may assume that $a = 0$.
In general form can be obtained by replacing $f\in (\Ind_{B}^{G}\chi)^{\sm,N\cap K,Z_{L}^{+} = \chi}$ with
\[
\begin{pmatrix}
1 & 0 & 0\\
a & 1 & 0\\
0 & 0 & 1
\end{pmatrix}^{-1}
f.
\]
This reduction step is obvious for the first and third formulas.
For the fifth formula, for $a\in \mathcal{O}_{F}\setminus\{0\}$, we use
\begin{align*}
\int_{(\varpi_{F})}f\left(\begin{pmatrix}
1 & 0 & 0\\
a & 1 & 0\\
0 & 0 & 1
\end{pmatrix}
\begin{pmatrix}
1 & t & 0\\
0 & 1 & 0\\
0 & 0 & 1
\end{pmatrix}
\dot{s}_{1}
\right)dt
& = 
\int_{(\varpi_{F})}f\left(\begin{pmatrix}
1 & t(1 + at)^{-1} & 0\\
0 & 1 & 0\\
0 & 0 & 1
\end{pmatrix}
\dot{s}_{1}
\begin{pmatrix}
1 + at & -a & 0\\
0 & (1 + at)^{-1} & 0\\
0 & 0 & 1
\end{pmatrix}
\right)dt\\
& = 
\int_{(\varpi_{F})}
\eta_{1}(1 + at)f\left(\begin{pmatrix}
1 & t(1 + at)^{-1} & 0\\
0 & 1 & 0\\
0 & 0 & 1
\end{pmatrix}
\dot{s}_{1}
\right)dt\\
& = 
\int_{1 + (a\varpi_{F})}
\eta_{1}(t_{1})f\left(\begin{pmatrix}
1 & a^{-1}(1 - t_{1}^{-1}) & 0\\
0 & 1 & 0\\
0 & 0 & 1
\end{pmatrix}
\dot{s}_{1}
\right)\lvert a\rvert_F^{-1}dt_{1}\\
& = 
\int_{1 + (a\varpi_{F})}
\eta_{1}(t_{2})^{-1}f\left(\begin{pmatrix}
1 & a^{-1}(1 - t_{2}) & 0\\
0 & 1 & 0\\
0 & 0 & 1
\end{pmatrix}
\dot{s}_{1}
\right)\lvert a\rvert_F^{-1}dt_{2}\\
& = 
\int_{(\varpi_{F})}
\eta_{1}(1 - at_{3})^{-1}f\left(\begin{pmatrix}
1 & t_{3} & 0\\
0 & 1 & 0\\
0 & 0 & 1
\end{pmatrix}
\dot{s}_{1}
\right)dt_{3},
\end{align*}
where $t_{1} := 1 + at$, $t_{2} := t_{1}^{-1}$ and $t_{3} := a^{-1}(1 - t_{2})$.
Finally notice that if $\delta_{c(\eta_{1}\eta_{2}) = 0} \ne 0$ then $\eta_{1}(1 - at_{3})^{-1} = \eta_{2}(1 - at_{3})$ for $t_{3}\in (\varpi_{F})$.

Fix $f\in (\Ind_{B}^{G}\chi)^{\sm,N\cap K,Z_{L}^{+} = \chi}$.

\begin{lem}\label{lem:s2s1-formula}
Let $x\in F$ and $D\subset F$ a compact subset such that $x\notin D$.
Then for sufficiently large $y$, we have
\[
f\left(\begin{pmatrix}
1 & 0 & y\\
x & 1 & ay\\
0 & 0 & 1
\end{pmatrix}\dot s_{2}\dot s_{1}\right)
=
\eta_1\eta_2(y)\eta_2(a - x)f\begin{pmatrix}
1 & 0 & 0\\
a & 1 & 0\\
0 & 0 & 1
\end{pmatrix}
\]
for all $a\in D$.
\end{lem}
\begin{proof}
By
\[
\begin{pmatrix}
1 & 0 & y\\
x & 1 & ay\\
0 & 0 & 1
\end{pmatrix}
\dot{s}_{2}\dot{s}_{1}
=
\begin{pmatrix}
1 & 0 & 0\\
a & 1 & 0\\
1/y & -1/(y(x - a)) & 1
\end{pmatrix}
\begin{pmatrix}
y & -1 & 0\\
0 & a - x & -1\\
0 & 0 & 1/(y(a - x))
\end{pmatrix},
\]
the left-hand side equals
\[
\eta_1\eta_2(y)\eta_2(a - x)f\begin{pmatrix}
1 & 0 & 0\\
a & 1 & 0\\
1/y & -1/(y(x - a)) & 1
\end{pmatrix}.
\]
Hence the lemma follows from the smoothness of $f$.
\end{proof}

\begin{lem}\label{lem:on s2s1 variant}
Let $x,y\in F$ such that $y + s \ne 0$, $x - t/(y + s)\ne 0$ for any $s,t\in \mathcal{O}_{F}$.
Then \begin{align*}
f\left(\begin{pmatrix}
1 & 0 & y\\
x & 1 & 0\\
0 & 0 & 1
\end{pmatrix}\dot s_{2}\dot s_{1}\right)
=
\int_{\mathcal{O}_{F}} \int_{\mathcal{O}_{F}} \eta_1\eta_2(y + s)\eta_2\left(\frac{t}{y + s} - x\right)
f\begin{pmatrix}
1 & 0 & 0\\
\frac{t}{y + s} & 1 & 0\\
0 & 0 & 1
\end{pmatrix}ds dt.
\end{align*}
\end{lem}
\begin{proof}
Let $z := \diag(\varpi_{F}^{k},\varpi_{F}^{k},1)$.
As $\tau_{z}f = \chi(z)f$, we have
\begin{align*}
& f\left(\begin{pmatrix}
1 & 0 & y\\
x & 1 & 0\\
0 & 0 & 1
\end{pmatrix}
\dot s_{2}\dot s_{1}\right)
= \chi_{1}\chi_{2}(\varpi_{F}^{-k})(\tau_{z}f)\left(\begin{pmatrix}
1 & 0 & y\\
x & 1 & 0\\
0 & 0 & 1
\end{pmatrix}
\dot s_{2}\dot s_{1}\right)\\
& = \chi_{1}\chi_{2}(\varpi_{F}^{-k})
\int_{\mathcal{O}_{F}} \int_{\mathcal{O}_{F}} 
f\left(z^{-1}\begin{pmatrix}
1 & 0 & s\\
0 & 1 & t\\
0 & 0 & 1
\end{pmatrix}
\begin{pmatrix}
1 & 0 & y\\
x & 1 & 0\\
0 & 0 & 1
\end{pmatrix}
\dot s_{2}\dot s_{1}\right)ds dt\\
& = 
\chi_{1}\chi_{2}(\varpi_{F}^{-k})\chi_{2}\chi_{3}(\varpi_{F}^k)
\int_{\mathcal{O}_{F}} \int_{\mathcal{O}_{F}} f\left(\begin{pmatrix}
1 & 0 & \varpi_{F}^{-k}(y + s)\\
x & 1 & \varpi_{F}^{-k}t\\
0 & 0 & 1
\end{pmatrix}\dot s_{2}\dot s_{1}\right) ds dt.
\end{align*}
By Lemma~\ref{lem:s2s1-formula}, if $k$ is sufficiently large, this is equal to
\begin{align*}
& \eta_1\eta_2(\varpi_{F}^k)\int_{\mathcal{O}_{F}} \int_{\mathcal{O}_{F}} \eta_1\eta_2(\varpi_{F}^{-k}(y + s))\eta_2\left(\frac{t}{y + s} - x\right)
f\begin{pmatrix}
1 & 0 & 0\\
\frac{t}{y + s} & 1 & 0\\
0 & 0 & 1
\end{pmatrix}ds dt\\
& = \int_{\mathcal{O}_{F}} \int_{\mathcal{O}_{F}} \eta_1\eta_2(y + s)\eta_2\left(\frac{t}{y + s} - x\right)
f\begin{pmatrix}
1 & 0 & 0\\
\frac{t}{y + s} & 1 & 0\\
0 & 0 & 1
\end{pmatrix} ds dt.
\end{align*}
We get the lemma.
\end{proof}

  We prove the first formula of Theorem~\ref{thm:explicit}.

Let $b,c\in (\varpi_{F})$ and assume that $bc \ne 0$.
By Lemma~\ref{lem:on s2s1 variant} we have \begin{align*}
f\begin{pmatrix}1 & 0 & 0\\
0 & 1 & 0\\
b & c & 1
\end{pmatrix}
& = 
f\left(\begin{pmatrix}
1 & 0 & 1/b\\
-b/c & 1 & 0\\
0 & 0 & 1
\end{pmatrix}
\dot s_{2}\dot s_{1}
\begin{pmatrix}
b & c & 1\\
0 & c/b & 1/b\\
0 & 0 & 1/c
\end{pmatrix}\right)\\
& = \eta_1(b)\eta_2(c)\int_{\mathcal{O}_{F}} \int_{\mathcal{O}_{F}} \eta_1\eta_2\left(\frac{1}{b} + s\right)\eta_2\left(\frac{t}{1/b + s}  + \frac{b}{c} \right)
f\begin{pmatrix}
1 & 0 & 0 \\
\frac{t}{1/b + s} & 1 & 0\\
0 & 0 & 1
\end{pmatrix}ds dt\\
& = \int_{\mathcal{O}_{F}} \int_{\mathcal{O}_{F}} \eta_1\eta_2\left(1 + bs\right)\eta_2\left(\frac cb\right)\eta_2\left(\frac{tb}{1 + sb} + \frac{b}{c}\right)
f\begin{pmatrix}
1 & 0 & 0 \\
\frac{tb}{1 + sb} & 1 & 0\\
0 & 0 & 1
\end{pmatrix}ds dt.\\
\noalign{\noindent Writing $\frac{tb}{1 + sb} = bt'$ (i.e.\ $t' = \frac{t}{1 + sb}$ and $dt' = dt$) we get}
& = \int_{\mathcal{O}_{F}} \int_{\mathcal{O}_{F}} \eta_1\eta_2\left(1 + bs\right)\eta_2\left(\frac cb\right)\eta_2\left(bt'+\frac bc\right)
f\begin{pmatrix}
1 & 0 & 0 \\
bt' & 1 & 0\\
0 & 0 & 1
\end{pmatrix}ds dt'\\
& = \int_{\mathcal{O}_{F}} \int_{\mathcal{O}_{F}} \eta_1\eta_2\left(1 + bs\right)\eta_2(1+ct)
f\begin{pmatrix}
1 & 0 & 0 \\
bt & 1 & 0\\
0 & 0 & 1
\end{pmatrix}ds dt\\
& = \left(\int_{\mathcal{O}_{F}} \eta_1\eta_2(1 + bs) ds\right) \left(\int_{\mathcal{O}_{F}} \eta_2(1 + ct)
f\begin{pmatrix}
1 & 0 & 0 \\
bt & 1 & 0\\
0 & 0 & 1
\end{pmatrix}dt\right).
\end{align*}
As $\int_{\mathcal{O}_{F}} \eta_1\eta_2(1 + bs) ds = 1$ if $\eta_1\eta_2$ is trivial on $1+b\mathcal{O}_{F}$ and zero otherwise, we get the first formula of Theorem~\ref{thm:explicit} when $bc \ne 0$.
By local constancy of $f$ we can see that it also holds when $bc = 0$.

We prove the third formula in Theorem~\ref{thm:explicit}.
Note that our formula does not depend on $\varpi_{F}$.
We take $\varpi_{F}$ such that $\eta_{2}(\varpi_{F}) \ne q$, which is possible as $\eta_{2} \ne \lvert \cdot\rvert_F^{-1}$.
Let $c\in (\varpi_{F})$.
For $z = \diag(\varpi_{F}^{k},\varpi_{F}^{k},1)$ with $k \ge 0$, we have
\begin{align*}
f\left(\begin{pmatrix}
1 & 0 & 0\\
0 & 1 & 0\\
c & 0 & 1
\end{pmatrix}
\dot s_{2}
\right)
&=
\chi(z)^{-1} (\tau_{z}f)\left(\begin{pmatrix}
1 & 0 & 0\\
0 & 1 & 0\\
c & 0 & 1
\end{pmatrix}
\dot s_{2}
\right)\\
&=
\chi_{2}^{-1}\chi_{3}(\varpi_{F}^{k})\int_{\mathcal{O}_{F}} \int_{\mathcal{O}_{F}} 
f\left(\begin{pmatrix}
1 + cv & 0 & \varpi_{F}^{-k}v\\
cw & 1 & \varpi_{F}^{-k}w\\
\varpi_{F}^{k}c & 0 & 1
\end{pmatrix}\dot s_{2}\right)dv dw.
\end{align*}
We have
\[
\begin{pmatrix}
1 + cv & 0 & \varpi_{F}^{-k}v\\
cw & 1 & \varpi_{F}^{-k}w\\
\varpi_{F}^{k}c & 0 & 1
\end{pmatrix}\dot s_{2}
=
\begin{pmatrix}
1 & 0 & 0\\
\frac{cw}{1 + cv} & 1 & \varpi_{F}^{-k}w\\
\frac{\varpi_{F}^{k}c}{1 + cv} & 0 & 1\\
\end{pmatrix}
\dot s_{2}
\begin{pmatrix}
1 + cv & \varpi_{F}^{-k}v & 0\\
0 & 1/(1 + cv) & 0\\
0 & 0 & 1
\end{pmatrix}.
\]
Hence
\begin{equation}\label{eq:second case(1) variant}
\begin{split}
f\left(\begin{pmatrix}
1 & 0 & 0\\
0 & 1 & 0\\
c & 0 & 1
\end{pmatrix}
\dot s_{2}
\right)
& =
\eta_2(\varpi_{F}^{k})
\int_{\mathcal{O}_{F}} \int_{\mathcal{O}_{F}} \eta_1(1 + cv)
f\left(
\begin{pmatrix}
1 & 0 & 0\\
\frac{cw}{1 + cv} & 1 & \varpi_{F}^{-k}w\\
\frac{\varpi_{F}^{k}c}{1 + cv} & 0 & 1
\end{pmatrix}\dot s_{2}
\right)dv dw.
\end{split}
\end{equation}
Note that, by construction, the integrand in \eqref{eq:second case(1) variant} only depends on $v$ and $w$ modulo $\varpi_F^k$.

The $w \equiv 0$ part of \eqref{eq:second case(1) variant} equals, for $k$ sufficiently large (by smoothness),
\begin{align*}
\frac{\eta_2(\varpi_{F}^{k})}{q^{k}}
\int_{\mathcal{O}_{F}} \eta_1(1 + cv)
f\left(
\begin{pmatrix}
1 & 0 & 0\\
0 & 1 & 0\\
\frac{\varpi_{F}^{k}c}{1 + cv} & 0 & 1
\end{pmatrix}\dot s_{2}
\right) dv
&=\frac{\eta_2(\varpi_{F}^{k})}{q^{k}}
\int_{\mathcal{O}_{F}} \eta_1(1 + cv)
f\left(
\begin{pmatrix}
1 & 0 & 0\\
0 & 1 & 0\\
0 & 0 & 1
\end{pmatrix}\dot s_{2}
\right) dv\\
&=
\delta_{\val(c) \ge c(\eta_1)}\frac{\eta_2(\varpi_{F}^{k})}{q^{k}}
f\left(
\begin{pmatrix}
1 & 0 & 0\\
0 & 1 & 0\\
0 & 0 & 1
\end{pmatrix}\dot s_{2}
\right).
\end{align*}
We calculate the $w\not\equiv 0$ part. For such $w$ we have
\[
\begin{pmatrix}
1 & 0 & 0\\
\frac{cw}{1 + cv} & 1 & \varpi_{F}^{-k}w\\
\frac{\varpi_{F}^{k}c}{1 + cv} & 0 & 1
\end{pmatrix}\dot s_{2}
=
\begin{pmatrix}
1 & 0 & 0\\
\frac{cw}{1 + cv} & 1 & 0\\
\frac{\varpi_{F}^{k}c}{1 + cv} & \varpi_{F}^{k}/w & 1
\end{pmatrix}
\begin{pmatrix}
1 & 0 & 0\\
0 & \varpi_{F}^{-k}w & -1\\
0 & 0 & \varpi_{F}^{k}/w
\end{pmatrix}.
\]
Note that $cw/(1 + cv) \in \mathcal{O}_{F}$, $\varpi_{F}^{k}c/(1 + cv),\varpi_{F}^{k}/w\in (\varpi_{F})$.
Hence by the first formula of Theorem~\ref{thm:explicit} for $k$ sufficiently large, we get
\[
f\left(\begin{pmatrix}
1 & 0 & 0\\
\frac{cw}{1 + cv} & 1 & \varpi_{F}^{-k}w\\
\frac{\varpi_{F}^{k}c}{1 + cv} & 0 & 1
\end{pmatrix}\dot s_{2}\right)
=
\eta_2(\varpi_{F}^{-k}w)
\int_{\mathcal{O}_{F}} \eta_2\left(1+\frac{\varpi_{F}^k}w t\right)
f\begin{pmatrix}
1 & 0 & 0\\
\frac{c(w + \varpi_{F}^{k}t)}{1 + cv} & 1 & 0\\
0 & 0 & 1
\end{pmatrix} dt.
\]
Hence the $w\not\equiv 0$ part of \eqref{eq:second case(1) variant} equals
\begin{align*}
&
\frac{1}{q^{k}}\sum_{w\in (\mathcal{O}_{F}/(\varpi_{F}^{k}))\setminus\{0\}}
\int_{\mathcal{O}_{F}} \eta_1(1+cv)
\int_{\mathcal{O}_{F}} \eta_2(w+\varpi_{F}^k t)
f\begin{pmatrix}
1 & 0 & 0\\
\frac{c(w + \varpi_{F}^{k}t)}{1 + cv} & 1 & 0\\
0 & 0 & 1
\end{pmatrix}dt dv\\
& =
\frac{1}{q^{k}}\sum_{w\in (\mathcal{O}_{F}/(\varpi_{F}^{k}))\setminus\{0\}}
\int_{\mathcal{O}_{F}} \eta_1(1+cv)
\int_{w+(\varpi_{F}^k)} \eta_2(1+cv)\eta_2(t')
f\begin{pmatrix}
1 & 0 & 0\\
ct' & 1 & 0\\
0 & 0 & 1
\end{pmatrix}\frac{dt'}{|\varpi_{F}^k|_{F}} dv,\\
\noalign{\noindent where $t' := \frac{w+\varpi_{F}^k t}{1+cv}$, hence}
& =
\left(\int_{\mathcal{O}_{F}} \eta_1\eta_2(1+cv) dv\right)
\left(\int_{\mathcal{O}_{F}\setminus (\varpi_{F}^k)} \eta_2(t')
f\begin{pmatrix}
1 & 0 & 0\\
ct' & 1 & 0\\
0 & 0 & 1
\end{pmatrix}dt'\right)\\
& = 
\delta_{\val(c) \ge c(\eta_1\eta_2)}
\left(\int_{\mathcal{O}_{F}\setminus (\varpi_{F}^k)} \eta_2(t')
f\begin{pmatrix}
1 & 0 & 0\\
ct' & 1 & 0\\
0 & 0 & 1
\end{pmatrix}dt'\right)\\
& = 
\delta_{\val(c) \ge c(\eta_1\eta_2)}
\left(\int_{\mathcal{O}_{F}} \eta_2(t)
\Bigg[f\begin{pmatrix}
1 & 0 & 0\\
ct & 1 & 0\\
0 & 0 & 1
\end{pmatrix}-f(1)\Bigg]dt + \left(\int_{\mathcal{O}_{F}\setminus (\varpi_{F}^k)} \eta_2(t)dt\right) 
f(1)\right).
\end{align*}

In particular, when $c = 0$ equation \eqref{eq:second case(1) variant} gives
\begin{equation*}
  f(\dot{s}_{2})
 \left(1-\frac{\eta_2(\varpi_{F})^{k}}{q^{k}}\right)
= \left(\int_{\mathcal{O}_{F}\setminus (\varpi_{F}^k)} \eta_2(t)dt\right) f(1).
\end{equation*}

\begin{lem}\label{lem:int of chi(s)}
We have
\[
\int_{\mathcal{O}_{F}\setminus (\varpi_{F}^k)} \eta_2(t)dt = \delta_{c(\eta_2)=0} \cdot \frac{q - 1}{q - \eta_{2}(\varpi_{F})}\left(1-\frac{\eta_2(\varpi_{F})^{k}}{q^{k}}\right).
\]
\end{lem}
\begin{proof}
Let $n$ be the valuation of $t$.
Then 
\begin{align*}
\int_{\mathcal{O}_{F}\setminus (\varpi_{F}^k)} \eta_2(t)dt 
& = \sum_{n=0}^{k-1} \int_{\varpi_{F}^n \mathcal{O}_{F}^{\times}} \eta_2(t)dt
  = \sum_{n=0}^{k-1} \left(\frac{\eta_2(\varpi_{F})}{q}\right)^n \int_{\mathcal{O}_{F}^{\times}} \eta_2(s)ds,\\
\noalign{\noindent where $t := \varpi_{F}^n s$ and $dt = |\varpi_{F}^n|_{F} ds$, so}
& = \delta_{c(\eta_2)=0} \cdot \frac{q-1}q \cdot \frac{1 - (q^{-1}\eta_2(\varpi_{F}))^{k}}{1 - q^{-1}\eta_2(\varpi_{F})},
\end{align*}
where we used our assumption that $\eta_2(\varpi_{F})\ne q$.
\end{proof}
By our assumption that $\eta_2(\varpi_{F})\ne q$ we can choose our sufficiently large $k$ such
that $\eta_2(\varpi_{F}^k)\ne q^{k}$ (equality cannot hold for two consecutive values of $k$). Hence we get
\begin{equation}\label{eq:2}
f\left(\dot{s}_{2}
\right)
=
\delta_{c(\eta_2)=0} \frac{q - 1}{q - \eta_2(\varpi_{F})}
f(1).
\end{equation}

In general, substituting \eqref{eq:2} into the $w \equiv 0$ part we get from our analysis of \eqref{eq:second case(1) variant} and Lemma~\ref{lem:int of chi(s)} that
\begin{align*}
f\left(\begin{pmatrix}
1 & 0 & 0\\
0 & 1 & 0\\
c & 0 & 1
\end{pmatrix}
\dot s_{2}
\right)
&=
\delta_{\val(c) \ge c(\eta_1)}
\delta_{c(\eta_2)=0} \frac{\eta_2(\varpi_{F})^{k}}{q^{k}} \frac{q - 1}{q - \eta_2(\varpi_{F})}
f(1)
\\
&\quad +\delta_{\val(c) \ge c(\eta_1\eta_2)}
\Bigg(\int_{\mathcal{O}_{F}} \eta_2(t)
\Bigg[f\begin{pmatrix}
1 & 0 & 0\\
ct & 1 & 0\\
0 & 0 & 1
\end{pmatrix}-f(1)\Bigg]dt\\
&\quad + \delta_{c(\eta_2)=0} \cdot \frac{q - 1}{q - \eta_{2}(\varpi_{F})}\left(1-\frac{\eta_2(\varpi_{F})^{k}}{q^{k}}\right) f(1)\Bigg)\\
&=
\delta_{\val(c) \ge c(\eta_1\eta_2)} \Bigg(
\int_{\mathcal{O}_{F}} \eta_2(t) \left[f\begin{pmatrix}1 & 0 & 0\\ ct & 1 & 0\\ 0 & 0 & 1\end{pmatrix} - f(1)\right] dt \\
& \quad + \delta_{c(\eta_2) = 0}\cdot \frac{q - 1}{q - \eta_{2}(\varpi_{F})}  f(1)\Bigg).
\end{align*}
Here we used $\delta_{\val(c)\ge c(\eta_{1})}\delta_{c(\eta_{2}) = 0} = \delta_{\val(c)\ge c(\eta_{1}\eta_{2})}\delta_{c(\eta_{2}) = 0}$ in the last line.
We get the third formula of Theorem~\ref{thm:explicit}.

Finally we prove the fifth formula of Theorem~\ref{thm:explicit}.
Note that our formula does not depend on $\varpi_{F}$.
We take our uniformizer $\varpi_{F}$ such that $\eta_{1}\eta_{2}(\varpi_{F})\ne q^{2}$, which is possible as $\eta_{1}\eta_{2} \ne \lvert \cdot\rvert_F^{-2}$.
For $z_{1} = \diag(\varpi_{F},\varpi_{F},1)$, we have as in the proof of Lemma~\ref{lem:on s2s1 variant},
\begin{align*}
f\left(\dot s_{2}\dot s_{1}\right)
& =
\chi(z_{1})^{-1}(\tau_{z_{1}}f)\left(\dot s_{2}\dot s_{1}\right)\\
& = \eta_1\eta_2(\varpi_{F})
\int_{\mathcal{O}_{F}} \int_{\mathcal{O}_{F}} 
f\left(\begin{pmatrix}1 & 0 & \varpi_{F}^{-1}v\\ 0 & 1 & \varpi_{F}^{-1}w \\ 0 & 0 & 1\end{pmatrix}\dot s_{2}\dot s_{1}\right) dv dw.
\end{align*}

We calculate the right-hand side, recalling that the integrand only depends on $v$, $w$ modulo $\varpi_F$.
\begin{enumerate}
\item $v, w \equiv 0 \pmod{\varpi_F}$.
We have
\[
\frac{\eta_1\eta_2(\varpi_{F})}{q^{2}}
f\left(\dot s_{2}\dot s_{1}\right).
\]
\item $v \not\equiv 0$, $w \equiv 0 \pmod{\varpi_F}$.
We have
\[
\begin{pmatrix}
1 & 0 & \varpi_{F}^{-1}v\\
0 & 1 & 0\\
0 & 0 & 1
\end{pmatrix}\dot s_{2}\dot s_{1}
=
\begin{pmatrix}
1 & 0 & 0\\
0 & 1 & 0\\
\varpi_{F} v^{-1} & 0 & 1
\end{pmatrix}\dot s_{2}
\begin{pmatrix}
\varpi_{F}^{-1}v & -1 & 0\\
0 & \varpi_{F} v^{-1} & 0\\
0 & 0 & 1
\end{pmatrix}.
\]
Hence, by using the third formula of Theorem~\ref{thm:explicit}, the right-hand side equals
\begin{align*}
& \frac{\eta_1\eta_2(\varpi_{F})\eta_1(\varpi_{F}^{-1})}{q}
\int_{\mathcal{O}_{F}^\times}
\eta_1(v)
f\left(
\begin{pmatrix}
1 & 0 & 0\\
0 & 1 & 0\\
\varpi_{F} v^{-1} & 0 & 1
\end{pmatrix}\dot s_{2}
\right) dv\\
& = 
\delta_{1 \ge c(\eta_1\eta_2)} \frac{\eta_2(\varpi_{F})}{q}
\int_{\mathcal{O}_{F}^\times}
\eta_1(v)
\Bigg(
\int_{\mathcal{O}_{F}} \eta_2(t) \left[f\begin{pmatrix}1 & 0 & 0\\ \varpi_{F} v^{-1} t & 1 & 0\\ 0 & 0 & 1\end{pmatrix} - f(1)\right] dt \\
& \hspace{6cm} + \delta_{c(\eta_2) = 0}\cdot \frac{q - 1}{q - \eta_{2}(\varpi_{F})}  f(1)\Bigg) dv.
\end{align*}
Hence, by letting $t' := \varpi_{F} v^{-1} t$ (so $dt' = |\varpi_{F}|_{F}dt$) and noting that $\eta_2(v) = 1$ if the final term contributes, we get
\begin{align*}
&\frac{\delta_{1 \ge c(\eta_1\eta_2)}}{q} \left(\int_{\mathcal{O}_{F}^\times} \eta_1\eta_2(v)dv\right) 
\Bigg(
\int_{(\varpi_{F})} \eta_2(t') \left[f\begin{pmatrix}1 & 0 & 0\\ t' & 1 & 0\\ 0 & 0 & 1\end{pmatrix} - f(1)\right] \frac{dt'}{|\varpi_{F}|_{F}} \\
& \hspace{6cm} + \delta_{c(\eta_2) = 0}\cdot \frac{q - 1}{q - \eta_{2}(\varpi_{F})} \eta_2(\varpi_{F}) f(1)\Bigg),\\
& = \delta_{c(\eta_1\eta_2)=0} \frac{q-1}{q}
\Bigg(
\int_{(\varpi_{F})} \eta_2(t') \left[f\begin{pmatrix}1 & 0 & 0\\ t' & 1 & 0\\ 0 & 0 & 1\end{pmatrix} - f(1)\right] dt' \\
& \hspace{6cm} + \delta_{c(\eta_2) = 0}\cdot \frac{q - 1}{q - \eta_{2}(\varpi_{F})} \frac{\eta_2(\varpi_{F})}{q} f(1)\Bigg).
\end{align*}
\item $v \equiv 0,w\not\equiv 0 \pmod{\varpi_F}$.
We have
\[
\begin{pmatrix}
1 & 0 & 0\\
0 & 1 & \varpi_{F}^{-1}w\\
0 & 0 & 1
\end{pmatrix}\dot s_{2}\dot s_{1}
=
\begin{pmatrix}
1 & 0 & 0\\
0 & 1 & 0\\
0 & \varpi_{F} w^{-1} & 1
\end{pmatrix}\dot s_{1}
\begin{pmatrix}
\varpi_{F}^{-1}w & 0 & -1\\
0 & 1 & 0\\
0 & 0 &  \varpi_{F} w^{-1}
\end{pmatrix}.
\]
Hence the right-hand side equals, by using the second formula of Theorem~\ref{thm:explicit},
\begin{align*}
\frac{\eta_1\eta_2(\varpi_{F})\eta_1\eta_2(\varpi_{F}^{-1})}{q}
& \int_{\mathcal{O}_{F}^\times} \eta_1\eta_2(w)f\left(
\begin{pmatrix}
1 & 0 & 0\\
0 & 1 & 0\\
0 & \varpi_{F} w^{-1} & 1
\end{pmatrix}\dot s_{1}\right) dw\\
& = 
  \frac{\delta_{1 \ge c(\eta_1\eta_2)}}{q}   \int_{\mathcal{O}_{F}^\times} \eta_1\eta_2(w)
  \displaystyle\int_{\mathcal{O}_{F}} f\left(\begin{pmatrix}1 & \varpi_{F} w^{-1}t & 0\\ 0 & 1 & 0\\ 0 & 0 & 1\end{pmatrix}\dot s_{1}\right)dt dw.
\end{align*}
By letting $t' := \varpi_{F} w^{-1}t$ (so $dt' = |\varpi_{F}|_{F}dt$), we get
\[
  \delta_{c(\eta_1\eta_2)=0}\frac{q-1}{q} 
  \displaystyle\int_{(\varpi_{F})} f\left(\begin{pmatrix}1 & t' & 0\\ 0 & 1 & 0\\ 0 & 0 & 1\end{pmatrix}\dot s_{1}\right)dt'.
\]
\item $v\not\equiv 0$, $w\not\equiv 0\pmod{\varpi_F}$.
We have
\[
\begin{pmatrix}
1 & 0 & \varpi_{F}^{-1}v\\
0 & 1 & \varpi_{F}^{-1}w\\
0 & 0 & 1
\end{pmatrix}\dot s_{2}\dot s_{1}
=
\begin{pmatrix}
1 & 0 & 0\\
v^{-1}w & 1 & 0\\
v^{-1}\varpi_{F} & \varpi_{F} w^{-1} & 1
\end{pmatrix}
\begin{pmatrix}
\varpi_{F}^{-1}v & -1 & 0\\
0 & wv^{-1} & -1\\
0 & 0 & \varpi_{F} w^{-1}
\end{pmatrix}.
\]
Hence we get
\begin{align*}
\frac{\eta_1\eta_2(\varpi_{F})}{q}\int_{\mathcal{O}_{F}^\times} 
\sum_{w\in (\mathcal{O}_{F}/(\varpi_{F}))\setminus\{0\}}
\eta_1\eta_2(\varpi_{F}^{-1})\eta_1(v)\eta_2(w)
f\begin{pmatrix}
1 & 0 & 0\\
v^{-1}w & 1 & 0\\
v^{-1}\varpi_{F} & \varpi_{F}w^{-1} & 1
\end{pmatrix}dv.
\end{align*}
Therefore, by applying the first formula of Theorem~\ref{thm:explicit}, we get
\[
\frac{\delta_{1 \ge c(\eta_1\eta_2)}}{q}\int_{\mathcal{O}_{F}^\times} 
\sum_{w\in (\mathcal{O}_{F}/(\varpi_{F}))\setminus\{0\}}
\eta_1(v)\eta_2(w)
\displaystyle\int_{\mathcal{O}_{F}} \eta_2(1+\varpi_{F}w^{-1}t) f\begin{pmatrix}1 & 0 & 0\\ v^{-1}(w + \varpi_{F} t) & 1 & 0\\ 0 & 0 & 1\end{pmatrix}dt dv.
\]
By letting $t' := v^{-1}(w+\varpi_{F} t)$ (so $dt' = |\varpi_{F}|_{F}dt$ and $vt' = w+\varpi_{F} t$), we get
\begin{align*}
\frac{\delta_{1 \ge c(\eta_1\eta_2)}}{q} &\int_{\mathcal{O}_{F}^\times} 
\sum_{w\in (\mathcal{O}_{F}/(\varpi_{F}))\setminus\{0\}}
\eta_1(v)
\displaystyle\int_{\mathcal{O}_{F}} \eta_2(w+\varpi_{F} t) f\begin{pmatrix}1 & 0 & 0\\ v^{-1}(w + \varpi_{F} t) & 1 & 0\\ 0 & 0 & 1\end{pmatrix}dt dv\\
&=\frac{\delta_{1 \ge c(\eta_1\eta_2)}}{q}\int_{\mathcal{O}_{F}^\times} 
\eta_1\eta_2(v) \sum_{w\in (\mathcal{O}_{F}/(\varpi_{F}))\setminus\{0\}}
\displaystyle\int_{v^{-1}w + (\varpi_{F})} \eta_2(t') f\begin{pmatrix}1 & 0 & 0\\ t' & 1 & 0\\ 0 & 0 & 1\end{pmatrix}\frac{dt'}{|\varpi_{F}|_{F}} dv\\  
&=\delta_{c(\eta_1\eta_2)=0}\frac{q-1}{q}
\displaystyle\int_{\mathcal{O}_{F}^{\times}} \eta_2(t') f\begin{pmatrix}1 & 0 & 0\\ t' & 1 & 0\\ 0 & 0 & 1\end{pmatrix}dt'.
\end{align*}
\end{enumerate}
Therefore,
\begin{align*}
\left(1 - \frac{\eta_1\eta_2(\varpi_{F})}{q^{2}}\right)
f\left(\dot s_{2}\dot s_{1}\right)
& = \frac{\delta_{c(\eta_1\eta_2)=0} (q-1)}{q}
\Bigg(
\int_{(\varpi_{F})} \eta_2(t') \left[f\begin{pmatrix}1 & 0 & 0\\ t' & 1 & 0\\ 0 & 0 & 1\end{pmatrix} - f(1)\right] dt' \\
& \hspace{4cm} + \delta_{c(\eta_2) = 0}\cdot \frac{q - 1}{q - \eta_{2}(\varpi_{F})} \frac{\eta_2(\varpi_{F})}q f(1)\Bigg)\\
& \quad + \frac{\delta_{c(\eta_1\eta_2)=0}(q-1)}{q} 
  \displaystyle\int_{(\varpi_{F})} f\left(\begin{pmatrix}1 & t' & 0\\ 0 & 1 & 0\\ 0 & 0 & 1\end{pmatrix}\dot s_{1}\right)dt'\\
& \quad + \frac{\delta_{c(\eta_1\eta_2)=0}(q-1)}{q}
\displaystyle\int_{\mathcal{O}_{F}^{\times}} \eta_2(t') f\begin{pmatrix}1 & 0 & 0\\ t' & 1 & 0\\ 0 & 0 & 1\end{pmatrix}dt'\\
& = \frac{\delta_{c(\eta_1\eta_2)=0} (q-1)}{q}
\Bigg(
\int_{\mathcal{O}_{F}} \eta_2(t') \left[f\begin{pmatrix}1 & 0 & 0\\ t' & 1 & 0\\ 0 & 0 & 1\end{pmatrix} - f(1)\right] dt' \\
& \hspace{4cm} + \delta_{c(\eta_2) = 0}\cdot \frac{q - 1}{q - \eta_{2}(\varpi_{F})} f(1)\Bigg)\\
& \quad + \frac{\delta_{c(\eta_1\eta_2)=0}(q-1)}{q} 
  \displaystyle\int_{(\varpi_{F})} f\left(\begin{pmatrix}1 & t' & 0\\ 0 & 1 & 0\\ 0 & 0 & 1\end{pmatrix}\dot s_{1}\right)dt',
\end{align*}
where we used Lemma~\ref{lem:int of chi(s)} (with $k = 1$) to combine the first and third terms. 
We obtain the fifth formula and hence conclude the proof of Theorem~\ref{thm:explicit}.

\subsection{Density argument}
\label{sec:density-arg}
\emph{We continue to assume that $\eta_{2}\ne \lvert\cdot\rvert_{F}^{-1}$ and $\eta_{1}\eta_{2}\ne\lvert\cdot\rvert_{F}^{-2}$.} For $n\ge 1$, $\gamma\in C^{\times}$ and a smooth function $g\colon F\to C$ that vanishes outside $\mathcal O_F$, we define $f_{n}\in (\Ind_{B}^{G}\chi)^{\sm,N\cap K,Z_{L}^{+} = \chi}$ as follows:
\begin{alignat*}{2}
f_{n}\begin{pmatrix}1 & 0 & 0\\ a & 1 & 0\\ 0 & 0 & 1\end{pmatrix} &= 0&&\qquad (a\in (\varpi_{F})),\\
f_{n}\left(\begin{pmatrix}1 & a & 0\\ 0 & 1 & 0\\ 0 & 0 & 1\end{pmatrix}\dot{s}_{1}\right) &= 
\gamma^{n}g\left(\frac{a}{\varpi_{F}^{n}}\right)&&\qquad (a\in \mathcal O_{F}).
\end{alignat*}
Here we use that the restriction map~\eqref{eq:jacquet module} is an isomorphism.
Let $\pi_{0}(g,\gamma)$ be the smallest closed subrepresentation of $(\Ind_{B}^{G}\chi)^{\cts}$ that contains all $f_{n}$ ($n \ge 1$).
(The functions $f_n$ depend on $g$ and $\gamma$, but we suppress this to keep the notation succinct.)

\begin{lem}\label{lem:f_n span dense subspace}
Assume that $\lvert \eta_{2}(\varpi_{F})\rvert > \lvert q\rvert$ and $\eta_2 \ne 1$ and $\eta_{1}\eta_{2}\ne\lvert\cdot\rvert_{F}^{-2}$.
There exist $g$ and $\gamma$ such that $\pi_{0}(g,\gamma)=(\Ind_{B}^{G}\chi)^{\cts}$ and $\int_{\mathcal{O}_{F}}g(t)dt = 0$.
\end{lem}
To prove the lemma, we calculate $f_{n}$ on each Iwahori orbit $I\dot{w}B$ ($w \in S_3$) with the help of Theorem~\ref{thm:explicit}.
(Recall that the Iwahori subgroup $I$ and the representatives $\dot w$ were defined in subsection~\ref{sec:expl-formula}.)

\noindent ($w = e$)
We have for $a,b,c\in (\varpi_{F})$:
\[
f_{n}\begin{pmatrix}1 & 0 & 0\\ a & 1 & 0 \\ b & c & 1\end{pmatrix}
 =
  \delta_{\val(b) \ge c(\eta_1\eta_2)}\displaystyle\int_{\mathcal{O}_{F}} \eta_2(1+ct) f_{n}\begin{pmatrix}1 & 0 & 0\\ a + bt & 1 & 0\\ 0 & 0 & 1\end{pmatrix}dt  = 0.
\]

\noindent ($w = s_{1}$)
We have for $a\in \mathcal{O}_{F}$ and $b,c\in (\varpi_{F})$:
\begin{align*}
f_{n}\left(\begin{pmatrix}1 & a & 0\\ 0 & 1 & 0 \\ b & c & 1\end{pmatrix}\dot s_{1}\right)
&=
  \delta_{\val(c) \ge c(\eta_1\eta_2)}
  \displaystyle\int_{\mathcal{O}_{F}} \eta_2(1+bt) f_{n}\left(\begin{pmatrix}1 & a + ct & 0\\ 0 & 1 & 0\\ 0 & 0 & 1\end{pmatrix}\dot s_{1}\right)dt \\
&= \delta_{\val(c) \ge c(\eta_1\eta_2)}\cdot\gamma^{n}\displaystyle\int_{\mathcal{O}_{F}} \eta_2(1+bt) g\left(\frac{a + ct}{\varpi_{F}^{n}}\right)dt.
\end{align*}
To simplify the notation, we put
\begin{equation}
  k_g(a,b,c) := \int_{\mathcal{O}_{F}}\eta_{2}(1 + at)g(b + ct)dt\label{eq:k_g}
\end{equation}
for $a\in (\varpi_{F})$ and $b,c\in F$.
Then
\[
f_{n}\left(\begin{pmatrix}1 & a & 0\\ 0 & 1 & 0 \\ b & c & 1\end{pmatrix}\dot s_{1}\right)
=\delta_{\val(c)\ge c(\eta_{1}\eta_{2})}\gamma^{n}k_g\left(b,\dfrac{a}{\varpi_{F}^{n}},\dfrac{c}{\varpi_{F}^{n}}\right).
\]

\noindent ($w = s_{2}$)
We have for $a,c \in (\varpi_F)$:
\begin{align*}
& f_{n}\left(\begin{pmatrix}1 & 0 & 0\\ a & 1 & 0\\ c & 0 & 1\end{pmatrix}\dot s_{2}\right)\\
& =
\delta_{\val(c) \ge c(\eta_1\eta_2)} \Bigg\{
\int_{\mathcal{O}_{F}} \eta_2(t) \left[f_{n}\begin{pmatrix}1 & 0 & 0\\ a+ct & 1 & 0\\ 0 & 0 & 1\end{pmatrix} - f_{n}\begin{pmatrix}1 & 0 & 0\\ a & 1 & 0\\ 0 & 0 & 1\end{pmatrix}\right] dt 
 + \delta_{c(\eta_2) = 0}\cdot \frac{q - 1}{q - \eta_{2}(\varpi_{F})}  f_{n}\begin{pmatrix}1 & 0 & 0\\ a & 1 & 0\\ 0 & 0 & 1\end{pmatrix}\Bigg\} = 0.
\end{align*}

\noindent ($w = s_{1}s_{2}$)
We have for $a \in \mathcal O_F$, $c \in (\varpi_F)$:
\begin{align*}
f_{n}\left(\begin{pmatrix}1 & a & 0\\ 0 & 1 & 0\\ 0 & c & 1\end{pmatrix}\dot s_{1}\dot s_{2}\right)
& =
\delta_{\val(c) \ge c(\eta_1\eta_2)} \cdot \gamma^n \Bigg\{
\int_{\mathcal{O}_{F}} \eta_2(t) \left(g\left(\frac{a - ct}{\varpi_{F}^{n}}\right) - g\left(\frac{a}{\varpi_{F}^{n}}\right)\right) dt \\
& \quad + \delta_{c(\eta_2) = 0}\cdot \frac{q - 1}{q - \eta_{2}(\varpi_{F})} g\left(\frac{a}{\varpi_{F}^{n}}\right) \Bigg\}.
\end{align*}
We calculate 
\[
\int_{\mathcal{O}_{F}} \eta_2(t) \left(g\left(\frac{a - ct}{\varpi_{F}^{n}}\right) - g\left(\frac{a}{\varpi_{F}^{n}}\right)\right) dt
\]
 for $a,c\in \mathcal{O}_{F}$ (not only for $c\in (\varpi_{F})$).

\begin{enumerate}
\item Assume $\val(a)<\val(c)$ and $\val(a) < n$.
Then $\val(a - ct) = \val(a) < n$ for any $t\in \mathcal{O}_{F}$.
Hence the value is zero.
\item 
Assume that $\val(a)\ge n$.
Then $\val(a - ct)\ge n$ if and only if $t\in (\varpi_{F}^{n}/c)$.
Hence if $\val(c) < n$ then the value is equal to the sum of
\begin{align*}
&\int_{(\varpi_{F}^{n}/c)}\eta_{2}(t)\left(g\left(\frac{a - ct}{\varpi_{F}^{n}}\right) - g\left(\frac{a}{\varpi_{F}^{n}}\right)\right)dt\\
& = \eta_{2}\left(\frac{\varpi_{F}^{n}}{c}\right)\left|\frac{\varpi_{F}^{n}}{c}\right|_{F}\int_{\mathcal{O}_{F}}\eta_{2}(t')\left(g\left(\frac{a}{\varpi_{F}^{n}} - t'\right) - g\left(\frac{a}{\varpi_{F}^{n}}\right)\right)dt'
\end{align*}
and
\begin{align*}
& \left(\int_{\mathcal{O}_{F} \setminus (\varpi_{F}^{n}/c)}\eta_{2}(t)dt\right)\left(- g\left(\frac{a}{\varpi_{F}^{n}}\right)\right)\\
& = - \delta_{c(\eta_2)=0} \cdot \frac{q - 1}{q - \eta_{2}(\varpi_{F})}\left(1-\frac{\eta_2(\varpi_{F})^{n-\val(c)}}{q^{n-\val(c)}}\right) g\left(\frac{a}{\varpi_{F}^{n}}\right)\\
& = - \delta_{c(\eta_2)=0} \cdot \frac{q - 1}{q - \eta_{2}(\varpi_{F})}\left(1-\eta_{2}\left(\frac{\varpi_{F}^{n}}{c}\right)\left|\frac{\varpi_{F}^{n}}{c}\right|_{F}\right) g\left(\frac{a}{\varpi_{F}^{n}}\right)
\end{align*}
by Lemma~\ref{lem:int of chi(s)}, noting that this expression is zero unless $\eta_2$ is unramified.
\item 
Otherwise $\val(c)\le\val(a)<n$. Then the value is equal to
\begin{align*}
\int_{a/c + (\varpi_{F}^{n}/c)}\eta_{2}(t)g\left(\frac{a - ct}{\varpi_{F}^{n}}\right)dt
& = \left|\frac{\varpi_{F}^{n}}{c}\right|_{F}\eta_{2}\left(\frac{a}{c}\right)\int_{\mathcal{O}_{F}}\eta_{2}\left(1 - \frac{\varpi_{F}^{n}}{a}t'\right)g(t')dt'\\
& =
\left|\frac{\varpi_{F}^{n}}{c}\right|_{F}\eta_{2}\left(\frac{a}{c}\right)k_g\left(-\frac{\varpi_{F}^{n}}{a},0,1\right).
\end{align*}
\end{enumerate}

To simplify the notation, we put
\begin{equation}
  h_g(a,b) := \int_{\mathcal{O}_{F}}\eta_{2}(t)\big(g(a + bt) - g(a)\big)dt + \delta_{c(\eta_{2}) = 0}\frac{q - 1}{q - \eta_{2}(\varpi_{F})}g(a)\label{eq:h_g}
\end{equation}
for $a, b \in \mathcal O_F$.
Then we have
\begin{align*}
& f_{n}\left(\begin{pmatrix}1 & a & 0\\ 0 & 1 & 0\\ 0 & c & 1\end{pmatrix}\dot s_{1}\dot s_{2}\right)\\
& =
\delta_{\val(c) \ge c(\eta_1\eta_2)}\cdot\gamma^{n}
\begin{cases}
h_g\left(\dfrac{a}{\varpi_{F}^{n}},-\dfrac{c}{\varpi_{F}^{n}}\right) & (\val(a)\ge n,\val(c)\ge n),\\
\eta_{2}\left(\dfrac{\varpi_{F}^{n}}{c}\right)\left|\dfrac{\varpi_{F}^{n}}{c}\right|_{F}h_g\left(\dfrac{a}{\varpi_{F}^{n}},-1\right) & (\val(a)\ge n,\val(c) < n),\\
\left|\dfrac{\varpi_{F}^{n}}{c}\right|_{F}\eta_{2}\left(\dfrac{a}{c}\right)k_g\left(-\dfrac{\varpi_{F}^{n}}{a},0,1\right) & (\val(a) < n,\val(c)\le \val(a)),\\
0 & (\val(a) < n,\val(c) > \val(a)).
\end{cases}
\end{align*}

\noindent ($w = s_{2}s_{1}$)
We have for $a \in (\varpi_F)$:
\begin{align*}
f_{n}\left(\begin{pmatrix}1 & 0 & 0\\ a & 1 & 0\\ 0 & 0 & 1\end{pmatrix}\dot s_{2}\dot s_{1}\right)
& = \frac{\delta_{c(\eta_1\eta_2)=0} (q-1)}{q(q^2-\eta_1\eta_2(\varpi_{F}))}
\Bigg\{
\int_{\mathcal{O}_{F}} \eta_2(t) \left[f_{n}\begin{pmatrix}1 & 0 & 0\\ a+t & 1 & 0\\ 0 & 0 & 1\end{pmatrix} - f_{n}\begin{pmatrix}1 & 0 & 0\\ a & 1 & 0\\ 0 & 0 & 1\end{pmatrix}\right] dt \\
& \quad + \delta_{c(\eta_2) = 0}\cdot \frac{q - 1}{q - \eta_{2}(\varpi_{F})} f_{n}\begin{pmatrix}1 & 0 & 0\\ a & 1 & 0\\ 0 & 0 & 1\end{pmatrix}
+ \displaystyle\int_{(\varpi_{F})} \eta_2(1-at) f_{n}\left(\begin{pmatrix}1 & t & 0\\ 0 & 1 & 0\\ 0 & 0 & 1\end{pmatrix}\dot s_{1}\right)dt\Bigg\}.
\end{align*}
The first two terms are zero, so
\begin{align*}
f_{n}\left(\begin{pmatrix}1 & 0 & 0\\ a & 1 & 0\\ 0 & 0 & 1\end{pmatrix}\dot s_{2}\dot s_{1}\right)
& =
\delta_{c(\eta_{1}\eta_{2}) = 0}\frac{q - 1}{q^{2} - \eta_{1}\eta_{2}(\varpi_{F})}\frac{\gamma^{n}}{q^{n+1}}\int_{\mathcal{O}_{F}}\eta_{2}(1 - a\varpi_{F}^{n}t)g(t)dt\\
& =
\delta_{c(\eta_{1}\eta_{2}) = 0}\frac{q - 1}{q^{2} - \eta_{1}\eta_{2}(\varpi_{F})}\frac{\gamma^{n}}{q^{n+1}}k_g(-a\varpi_{F}^{n},0,1).
\end{align*}

\noindent ($w = w_{0}$)
We have for $a \in \mathcal O_F$:
\begin{align*}
f_{n}\left(\begin{pmatrix}1 & a & 0\\ 0 & 1 & 0\\ 0 & 0 & 1\end{pmatrix}\dot w_{0}\right)
& = \frac{\delta_{c(\eta_1\eta_2)=0} (q-1)}{q(q^2-\eta_1\eta_2(\varpi_{F}))}\gamma^{n}
\Bigg\{
\int_{\mathcal{O}_{F}} \eta_2(t) \left(g\left(\frac{a - t}{\varpi_{F}^{n}}\right) - g\left(\frac{a}{\varpi_{F}^{n}}\right)\right) dt \\
& \quad + \delta_{c(\eta_2) = 0}\cdot \frac{q - 1}{q - \eta_{2}(\varpi_{F})} g\left(\frac{a}{\varpi_{F}^{n}}\right)
\Bigg\}.
\end{align*}
We have already calculated the term in parentheses.
We have
\begin{align*}
& f_{n}\left(\begin{pmatrix}1 & a & 0\\ 0 & 1 & 0\\ 0 & 0 & 1\end{pmatrix}\dot w_{0}\right)\\
& =
\delta_{c(\eta_{1}\eta_{2}) = 0}\frac{q - 1}{q(q^{2} - \eta_{1}\eta_{2}(\varpi_{F}))}\gamma^{n}
\begin{cases}
\eta_{2}(\varpi_{F}^{n})\lvert\varpi_{F}^{n}\rvert_{F}h_g\left(\dfrac{a}{\varpi_{F}^{n}},-1\right) & (\val(a) \ge n),\\
\eta_{2}(a)\lvert\varpi_{F}^{n}\rvert_{F}k_g\left(-\dfrac{\varpi_{F}^{n}}{a},0,1\right) & (\val(a) < n).
\end{cases}
\end{align*}

\begin{lem}\label{lem:how to take g}
Assume that $\eta_{2}\ne 1$ and $\eta_{2}\ne \lvert\cdot\rvert_{F}^{-1}$.
\begin{enumerate}
\item 
  For any smooth function $g\colon F\to C$ that vanishes outside $\mathcal O_F$ the functions $h_g$ \eqref{eq:h_g} and $k_g$ \eqref{eq:k_g} are smooth. In particular, $h_g$ is bounded on $\mathcal{O}_{F}\times \mathcal{O}_{F}$.
\item There exists a smooth function $g\colon F\to C$ that vanishes outside $\mathcal O_F$ such that $\int_{\mathcal{O}_{F}}g(t)dt = 0$ and $h_g(0,-1) \ne 0$.
\item Assume that $\int_{\mathcal{O}_{F}}g(t)dt = 0$.
Then the function $k_g(a,b,c)$ is compactly supported, hence bounded on $(\varpi_{F})\times F\times F$.
Moreover, $k_g(a,0,1) = 0$ if $\val(a) \ge c(\eta_2)$. \end{enumerate}
\end{lem}
\begin{proof}
For (i), take $\ell\in\Z_{>0}$ such that $g(x + t) = g(x)$ for any $t\in (\varpi_{F}^{\ell})$.
Then we have $h_g(a + a_{1},b + b_{1}) = h_g(a,b)$ for any $a_{1},b_{1}\in (\varpi_{F}^{\ell})$.
Hence $h_g$ is smooth and a similar argument applies for $k_g$.

We prove (ii).
Set $c := \max(c(\eta_{2}),1)$ and define $g\colon \mathcal{O}_{F}\to C$ by 
\[
g(x)
=
\begin{cases}
-1 & (x\in \mathcal{O}_{F}\setminus (-1 + (\varpi_{F}^{c}))),\\
q^{c} - 1 & (x\in -1 + (\varpi_{F}^{c})).
\end{cases}
\]
Then $g(x + y) = g(x)$ for any $y\in (\varpi_{F}^{c})$ and $\sum_{x\in \mathcal{O}_{F}/(\varpi_{F}^{c})}g(x) = 0$.
Hence $\int_{\mathcal{O}_{F}}g(x)dx = 0$.
We have
\begin{align*}
h_g(0,-1) & = \int_{\mathcal{O}_{F}}\eta_{2}(t)(g(-t) - g(0))dt + \delta_{c(\eta_{2}) = 0}\frac{q - 1}{q - \eta_{2}(\varpi_{F})}g(0)\\
& = \sum_{a\in \mathcal{O}_{F}/(\varpi_{F}^{c})}\int_{(\varpi_{F}^{c})}\eta_{2}(t - a)(g(a-t) - g(0))dt + \delta_{c(\eta_{2}) = 0}\frac{q - 1}{q - \eta_{2}(\varpi_{F})}g(0)\\
& = \sum_{a\in \mathcal{O}_{F}/(\varpi_{F}^{c})}(g(a) - g(0))\int_{(\varpi_{F}^{c})}\eta_{2}(t - a)dt + \delta_{c(\eta_{2}) = 0}\frac{q - 1}{q - \eta_{2}(\varpi_{F})}g(0)\\
& = q^{c}\int_{(\varpi_{F}^{c})}\eta_{2}(t + 1)dt - \delta_{c(\eta_{2}) = 0}\frac{q - 1}{q - \eta_{2}(\varpi_{F})}.
\end{align*}
We have $\eta_{2}(1 + t) = 1$ for any $t\in \varpi_{F}^{c}\mathcal{O}_{F}$.
Hence
\[
h_g(0,-1) = 1 - \delta_{c(\eta_{2}) = 0}\frac{q - 1}{q - \eta_{2}(\varpi_{F})}
=\begin{cases}
\dfrac{1 - \eta_{2}(\varpi_{F})}{q - \eta_{2}(\varpi_{F})} & (c(\eta_{2}) = 0),\\
1 & (c(\eta_{2}) > 0).
\end{cases}
\]
This is not zero, as $\eta_2 \ne 1$ (here we finally use this assumption).

Consider $k_g(a,b,c) = \int_{\mathcal{O}_{F}}\eta_{2}(1 + at)g(b + ct)dt$ and recall that $\supp(g) \subset \mathcal{O}_F$.
The $p$-adic balls $b+c\mathcal{O}_F$ and $\mathcal{O}_F$ are either disjoint or nested.
If $(b+c\mathcal{O}_F)\cap \mathcal{O}_F = \varnothing$, then $k_g(a,b,c) = 0$, and $b+c\mathcal{O}_F \subset \mathcal{O}_F$ is equivalent to $b,c \in \mathcal{O}_F$ (compact).
Hence it remains to consider the case where $\mathcal{O}_F \subset b+c\mathcal{O}_F$, or equivalently $\val(c) \le \min(0,\val(b))$.
Letting $t' := b+ct$ we obtain
\begin{align*}
  k_g(a,b,c) &= \lvert c\rvert_{F}^{-1} \int_{\mathcal{O}_{F}}\eta_{2}\left(1 + a\cdot \dfrac{t'-b}c\right)g(t')dt' \\
  &= \lvert c\rvert_{F}^{-1} \eta_2\left(1 - \dfrac{ab}c\right) \int_{\mathcal{O}_{F}}\eta_{2}\left(1 + \dfrac{a}{c-ab} t'\right)g(t')dt'.
\end{align*}
Hence if $\val(a)-\val(c-ab) \ge c(\eta_2)$, then $k_g(a,b,c) = 0$.
Note that $\val(c-ab) = \val(c)$, as $\val(ab) > \val(b) \ge \val(c)$.
Therefore in this region, $k_g$ is supported on the compact subset $\val(b) \ge \val(c) > \val(a)-c(\eta_2) > -c(\eta_2)$.
\end{proof}

\begin{proof}[Proof of Lemma~\ref{lem:f_n span dense subspace}]
We assume that $g$ satisfies the condition of Lemma~\ref{lem:how to take g}(ii) and prove that $\pi_{0}(g,\gamma)=(\Ind_{B}^{G}\chi)^{\cts}$ for a suitable $\gamma \in C^\times$.
Let us assume that $n \ge c(\eta_2)$ from now on.
Since $\int_{\mathcal{O}_{F}}g(x)dx = 0$, the formula for $f_{n}$ simplifies and we have the following, where $a\in \mathcal{O}_{F}$ and $b,c\in (\varpi_{F})$:
\begin{align*}
f_{n}\left(\begin{pmatrix}1 & a & 0\\ 0 & 1 & 0 \\ b & c & 1\end{pmatrix}\dot s_{1}\right)
& = 
\delta_{\val(c) \ge c(\eta_1\eta_2)}\cdot \gamma^{n}k_g\left(b,\frac{a}{\varpi_{F}^{n}},\frac{c}{\varpi_{F}^{n}}\right),\\
f_{n}\left(\begin{pmatrix}1 & a & 0\\ 0 & 1 & 0\\ 0 & c & 1\end{pmatrix}\dot s_{1}\dot s_{2}\right)
& =
\delta_{\val(c) \ge c(\eta_1\eta_2)}\cdot\gamma^{n}
\begin{cases}
h_g\left(\dfrac{a}{\varpi_{F}^{n}},-\dfrac{c}{\varpi_{F}^{n}}\right) & (\val(a)\ge n,\val(c)\ge n),\\
\eta_{2}\left(\dfrac{\varpi_{F}^{n}}{c}\right)\left|\dfrac{\varpi_{F}^{n}}{c}\right|_{F}h_g\left(\dfrac{a}{\varpi_{F}^{n}},-1\right) & (\val(a)\ge n,\val(c) < n),\\
\left|\dfrac{\varpi_{F}^{n}}{c}\right|_{F}\eta_{2}\left(\dfrac{a}{c}\right)k_g\left(-\dfrac{\varpi_{F}^{n}}{a},0,1\right)
& (\val(a) < n,\val(c)\le \val(a)),\\
0 & (\val(a) < n,\val(c) > \val(a)),
\end{cases}\\
f_{n}\left(\begin{pmatrix}1 & a & 0\\ 0 & 1 & 0\\ 0 & 0 & 1\end{pmatrix}\dot w_{0}\right)
& =
\delta_{c(\eta_{1}\eta_{2}) = 0}\frac{q - 1}{q(q^{2} - \eta_{1}\eta_{2}(\varpi_{F}))}\gamma^{n}
\begin{cases}
\eta_{2}(\varpi_{F}^{n})\lvert\varpi_{F}^{n}\rvert_{F}h_g\left(\dfrac{a}{\varpi_{F}^{n}},-1\right) & (\val(a) \ge n),\\
\eta_{2}(a)\lvert\varpi_{F}^{n}\rvert_{F}k_g\left(-\dfrac{\varpi_{F}^{n}}{a},0,1\right) & (\val(a) < n).
\end{cases}
\end{align*}
On the other orbits, $f_n$ vanishes.
(In case of $w = s_2 s_1$, this is because $n \ge c(\eta_2)$ and Lemma~\ref{lem:how to take g}(iii).)

\emph{From now on, we use the assumption $\lvert \eta_{2}(\varpi_{F})\rvert > \lvert q\rvert$ and put $\gamma := q/\eta_{2}(\varpi_{F})$, so $|\gamma| < 1$.}
We define $v\in (\Ind_{B\cap L}^{L}\chi)^\cts$ by
\[
v\begin{pmatrix}
1 & 0 & 0\\
0 & 1 & 0\\
0 & 0 & 1
\end{pmatrix}
=
0,\quad
v\begin{pmatrix}
1 & 0 & 0\\
a & 1 & 0\\
0 & 0 & 1
\end{pmatrix}
=
\delta_{\val(a) \ge c(\eta_1\eta_2)}
\frac{q^{\val(a)}}{\eta_{2}(a)},\quad
v\left(\begin{pmatrix}
1 & a & 0\\
0 & 1 & 0\\
0 & 0 & 1
\end{pmatrix}\dot s_{1}
\right)
=
\delta_{c(\eta_{1}\eta_{2}) = 0}\frac{q - 1}{q(q^{2} - \eta_{1}\eta_{2}(\varpi_{F}))},
\]
where $a \in (\varpi_{F})\setminus\{0\}$, respectively $a \in \mathcal{O}_{F}$.
Note that this defines a continuous function since $\lim_{a\to 0}\lvert q^{\val(a)}/\eta_{2}(a)\rvert = \lim_{a\to 0}\lvert \gamma\rvert^{\val(a)} = 0$.
We also define $h'_{n}\in C^0(\overline{N},Cv)$ by
\[
h'_{n}\begin{pmatrix}
1 & 0 & 0\\
0 & 1 & 0\\
b & c & 1
\end{pmatrix}
= 
\begin{cases}
h_g\left(\dfrac{b}{\varpi_{F}^{n}},-1\right)v & (\val(b)\ge n,c\in \mathcal{O}_F),\\
\eta_{2}\left(\dfrac{b}{\varpi_{F}^{n}}\right)k_g\left(-\dfrac{\varpi_{F}^{n}}{b},0,1\right)v
 & (0 \le \val(b) < n,c\in \mathcal{O}_F),\\
0 & (\text{otherwise}).
\end{cases}
\]
This function can be regarded as an element of $(\Ind_{P}^{G}(\Ind_{B\cap L}^{L}\chi)^{\cts})^{\cts}$ as usual: $h'_{n}(\overline{n}p) = p^{-1}h'_{n}(\overline{n})\in (\Ind_{B\cap L}^{L}\chi)^{\cts}$ for $p\in P$, $\overline n \in \overline N$ and $h'_{n}|_{G\setminus \overline{N}P} = 0$.
Since $(\Ind_{P}^{G}(\Ind_{B\cap L}^{L}\chi)^{\cts})^{\cts}\simeq (\Ind_{B}^{G}\chi)^{\cts}$
we can regard $h'_{n}$ as an element of $(\Ind_{B}^{G}\chi)^{\cts}$, namely the function $x\mapsto h_n'(x)(1)$ for $x \in G$.
A concrete description of $h'_{n}$ is as follows.
First we have 
\begin{align*}
\supp(h'_{n}) \subset \begin{pmatrix}
1 & 0 & 0\\
0 & 1 & 0\\
\mathcal{O}_{F} & \mathcal{O}_{F} & 1
\end{pmatrix}
P/B
&\subset I^{s_{1}s_{2}}\big((I^{s_{1}s_{2}} \cap L)(B\cap L)\cup (I^{s_{1}s_{2}} \cap L)\dot s_{1}(B\cap L)\big)B/B \\
&\subset I^{s_{1}s_{2}}B/B\cup I^{s_{1}s_{2}}\dot s_{1}B/B,
\end{align*}
where $I^{s_{1}s_{2}} := (\dot s_{1}\dot s_{2})^{-1}I\dot s_{1}\dot s_{2}$ is another Iwahori subgroup.
On each orbit we have the following.
If $a\in (\varpi_{F}),b,c\in \mathcal{O}_{F}$ then
\begin{align*}
& h'_{n}\begin{pmatrix}
1 & 0 & 0\\
a & 1 & 0\\
b & c & 1
\end{pmatrix}\\
& =
h'_{n}\left(\begin{pmatrix}
1 & 0 & 0\\
0 & 1 & 0\\
b - ac & c & 1
\end{pmatrix}
\begin{pmatrix}
1 & 0 & 0\\
a & 1 & 0\\
0 & 0 & 1
\end{pmatrix}\right)\\
& =
\delta_{\val(a) \ge c(\eta_1\eta_2)}
\begin{cases}
h_g\left(\dfrac{b - ac}{\varpi_{F}^{n}},-1\right)\dfrac{q^{\val(a)}}{\eta_{2}(a)} & (\val(b - ac)\ge n),\\
\eta_{2}\left(\dfrac{b - ac}{\varpi_{F}^{n}}\right)k_g\left(-\dfrac{\varpi_{F}^{n}}{b - ac},0,1\right)\dfrac{q^{\val(a)}}{\eta_{2}(a)} & (\val(b - ac) < n),
\end{cases}
\end{align*}
where we interpret $\frac{q^{\val(a)}}{\eta_{2}(a)}$ as $0$ when $a = 0$.
If $a,b,c\in \mathcal{O}_{F}$, then we have
\begin{align*}
& h'_{n}\left(
\begin{pmatrix}
1 & a & 0\\
0 & 1 & 0\\
c & b & 1
\end{pmatrix}\dot s_{1}\right)\\
& =
h'_{n}\left(
\begin{pmatrix}
1 & 0 & 0\\
0 & 1 & 0\\
c & b-ac & 1
\end{pmatrix}
\begin{pmatrix}
1 & a & 0\\
0 & 1 & 0\\
0 & 0 & 1
\end{pmatrix}\dot s_{1}\right)\\
& =
\delta_{c(\eta_{1}\eta_{2}) = 0}\frac{q - 1}{q(q^{2} - \eta_{1}\eta_{2}(\varpi_{F}))}
\begin{cases}
h_g\left(\dfrac{c}{\varpi_{F}^{n}},-1\right) & (\val(c)\ge n),\\
\eta_{2}\left(\dfrac{c}{\varpi_{F}^{n}}\right)k_g\left(-\dfrac{\varpi_{F}^{n}}{c},0,1\right)
& (\val(c) < n).
\end{cases}
\end{align*}

Put $h_{n} := (\dot{s}_{1}\dot{s}_{2})^{-1}f_{n} \in \pi_0(g,\gamma)$.
The value on $I^{s_{1}s_{2}}B/B$ is as follows.
If $a\in (\varpi_{F}),b,c\in \mathcal{O}_{F}$, then
\begin{align*}
& h_{n}\begin{pmatrix}
1 & 0 & 0\\
a & 1 & 0\\
b & c & 1
\end{pmatrix}\\
& = 
f_{n}\left(\begin{pmatrix}
1 & b & c\\
0 & 1 & 0\\
0 & a & 1
\end{pmatrix}\dot{s}_{1}\dot{s}_{2}\right)\\
& = 
f_{n}\left(\begin{pmatrix}
1 & 0 & c\\
0 & 1 & 0\\
0 & 0 & 1
\end{pmatrix}\begin{pmatrix}
1 & b - ac & 0\\
0 & 1 & 0\\
0 & a & 1
\end{pmatrix}\dot{s}_{1}\dot{s}_{2}\right)\\
& = 
f_{n}\left(\begin{pmatrix}
1 & b - ac & 0\\
0 & 1 & 0\\
0 & a & 1
\end{pmatrix}\dot{s}_{1}\dot{s}_{2}\right)\\
& =
\delta_{\val(a) \ge c(\eta_1\eta_2)}
\begin{cases}
\gamma^{n}h_g\left(\dfrac{b - ac}{\varpi_{F}^{n}},-\dfrac{a}{\varpi_{F}^{n}}\right) & (\val(b - ac)\ge n,\val(a)\ge n),\\
\dfrac{q^{\val(a)}}{\eta_{2}(a)}h_g\left(\dfrac{b - ac}{\varpi_{F}^{n}},-1\right) & (\val(b - ac)\ge n,\val(a) < n),\\
\dfrac{q^{\val(a)}}{\eta_{2}(a)}\eta_{2}\left(\dfrac{b - ac}{\varpi_{F}^{n}}\right)
k_g\left(-\dfrac{\varpi_{F}^{n}}{b - ac},0,1\right)
 & (\val(b - ac) < n,\val(a)\le \val(b - ac)),\\
0 & (\val(b - ac) < n,\val(a) > \val(b - ac)).
\end{cases}
\end{align*}
On $I^{s_{1}s_{2}}\dot s_{1}B/B$, it is as follows.
Let $a,b,c\in \mathcal{O}_{F}$.
Then
\begin{align*}
h_{n}\left(\begin{pmatrix}
1 & a & 0\\
0 & 1 & 0\\
c & b & 1
\end{pmatrix}
\dot s_{1}\right) & = 
f_{n}\left(
\begin{pmatrix}
1 & c & b\\
0 & 1 & a\\
0 & 0 & 1
\end{pmatrix}\dot w_{0}
\right)\\
& = f_{n}\left(
\begin{pmatrix}
1 & 0 & b\\
0 & 1 & a\\
0 & 0 & 1
\end{pmatrix}
\begin{pmatrix}
1 & c & 0\\
0 & 1 & 0\\
0 & 0 & 1
\end{pmatrix}\dot w_{0}
\right)\\
& = f_{n}\left(
\begin{pmatrix}
1 & c & 0\\
0 & 1 & 0\\
0 & 0 & 1
\end{pmatrix}\dot w_{0}
\right)\\
& =
\delta_{c(\eta_{1}\eta_{2}) = 0}\frac{q - 1}{q(q^{2} - \eta_{1}\eta_{2}(\varpi_{F}))}
\begin{cases}
h_g\left(\dfrac{c}{\varpi_{F}^{n}},-1\right) & (\val(c) \ge n),\\
\eta_{2}\left(\dfrac{c}{\varpi_{F}^{n}}\right)
k_g\left(-\dfrac{\varpi_{F}^{n}}{c},0,1\right)
& (\val(c) < n),
\end{cases}
\end{align*}
which is the same as the value of $h'_{n}$.

To prove that $\pi_{0}(g,\gamma)=(\Ind_{B}^{G}\chi)^{\cts}$ we apply Corollary~\ref{cor:irreducibility criterion} with $\overline N_0 = \overline N \cap K$, where we think of $(\Ind_{B}^{G}\chi)^{\cts}$ as $(\Ind_{P}^{G}(\Ind_{B\cap L}^{L}\chi)^{\cts})^{\cts}$.
As indicated in \S\ref{sec:notation}, our choice of norm on $C$ induces an $L \cap K$-invariant norm on $(\Ind_{B\cap L}^{L}\chi)^{\cts}$, which in turn induces a $K$-invariant norm $\lVert\cdot\rVert$ on $(\Ind_{P}^{G}(\Ind_{B\cap L}^{L}\chi)^{\cts})^{\cts}$.
This norm $\lVert\cdot\rVert$ is nothing but the $K$-invariant norm induced on $(\Ind_{B}^{G}\chi)^{\cts}$ by our norm on $C$.

We have
\[
\inf_{n}\lVert h'_{n}\rVert\ge \inf_{n}\left|h'_{n}\begin{pmatrix}1 & 0 & 0\\ \varpi_F^{c(\eta_{1}\eta_{2}) + 1} & 1 & 0\\ 0 & 0 & 1\end{pmatrix}\right| = \left| h_g(0,-1)\right|\lvert \gamma\rvert^{c(\eta_{1}\eta_{2}) + 1} > 0,
\]
as $h_g(0,-1) \ne 0$.
Hence it is sufficient to prove $\lim_{n\to\infty}(h_{n} - h'_{n}) = 0$.
On $I^{s_{1}s_{2}}\dot{s}_{1}\dot{s}_{2}B/B\cup I^{s_{1}s_{2}}\dot{s}_{2}\dot{s}_{1}B/B\cup I^{s_{1}s_{2}}\dot{w}_{0}B/B$, $h_{n} = h'_{n} = 0$.
On $I^{s_{1}s_{2}}\dot{s}_{1}B/B$, we have $h_{n} = h'_{n}$.
On $I^{s_{1}s_{2}}\dot{s}_{2}B/B$, we have $h'_{n} = 0$ and it is sufficient to prove $\lim_{n\to\infty}\lVert h_{n}\rVert = 0$.
This is equivalent to $\lim_{n\to\infty}\lVert f_{n}\rVert = 0$ on $I\dot{s}_{1}B/B$, and this is true since the function $k_g$ is bounded and $\lvert\gamma\rvert < 1$.

Finally we estimate
\begin{equation}\label{eq:h-h',last case}
\left|
h_{n}\begin{pmatrix}
1 & 0 & 0\\
a & 1 & 0\\
b & c & 1
\end{pmatrix}
-
h'_{n}\begin{pmatrix}
1 & 0 & 0\\
a & 1 & 0\\
b & c & 1
\end{pmatrix}
\right|.
\end{equation}
for $a\in (\varpi_{F})$ and $b,c\in \mathcal{O}_{F}$.
If $\val(a) \le \val(b - ac) < n$ or $\val(a) < n \le \val(b - ac)$, then this is zero.
If $\val(b - ac) < n$ and $\val(a) > \val(b - ac)$, then \eqref{eq:h-h',last case} equals
\[
\left|
\dfrac{q^{\val(a)}}{\eta_{2}(a)}\eta_{2}\left(\dfrac{b - ac}{\varpi_{F}^{n}}\right)k_g\left(-\dfrac{\varpi_{F}^{n}}{b - ac},0,1\right)\right|
= \lvert \gamma\rvert^{\val(a)} \left|\eta_{2}\left(\dfrac{b - ac}{\varpi_{F}^{n}}\right)k_g\left(-\dfrac{\varpi_{F}^{n}}{b - ac},0,1\right)\right|.
\]
This is zero if $n \ge \val(b - ac) + c(\eta_{2})$ by Lemma~\ref{lem:how to take g}(iii).
Assume $n < \val(b - ac) + c(\eta_{2})$.
Then $\val(a) > n - c(\eta_{2})$.
As $\lvert\gamma\rvert < 1$, we get at most
\begin{align*}& \le \lvert \gamma\rvert^{n - c(\eta_{2})} \left|\eta_{2}(\varpi_F)\right|^{\val((b-ac)/\varpi_{F}^{n})} \left|k_g\left(-\dfrac{\varpi_{F}^{n}}{b - ac},0,1\right)\right|\notag\\
& \le \lvert \gamma\rvert^{n - c(\eta_{2})} \left|\eta_{2}(\varpi_F)\right|^{\val((b-ac)/\varpi_{F}^{n})} \sup_{x\in (\varpi_{F})}\lvert k_g(x,0,1)\rvert.
\end{align*}
We have $-c(\eta_{2}) < \val((b - ac)/\varpi_{F}^{n}) < 0$.
Therefore there exists $r > 0$ which does not depend on $a,b,c,n$ such that \eqref{eq:h-h',last case} is less than or equal to $r\lvert \gamma\rvert^{n}$ in this case.

If $\val(b - ac) \ge n$ and $\val(a)\ge n$, then 
\[
\left|
h_{n}
\begin{pmatrix}
1 & 0 & 0\\
a & 1 & 0\\
b & c & 1
\end{pmatrix}
\right|
\le
\lvert \gamma\rvert^{n}\sup_{x,y\in \mathcal{O}_{F}}\lvert h_g(x,y)\rvert
\]
and
\[
\left|
h'_{n}
\begin{pmatrix}
1 & 0 & 0\\
a & 1 & 0\\
b & c & 1
\end{pmatrix}
\right|
\le 
\left|\dfrac{q^{\val(a)}}{\eta_{2}(a)}\right|\sup_{x\in \mathcal{O}_{F}}\lvert h_g(x,-1)\rvert.
\]
Since $\val(a)\ge n$ and $\lvert\gamma\rvert < 1$, we have $\lvert q^{\val(a)}/\eta_{2}(a)\rvert = \lvert\gamma\rvert^{\val(a)} \le \lvert \gamma\rvert^{n}$.
Hence \eqref{eq:h-h',last case} is less than or equal to $\lvert \gamma\rvert^{n}\sup_{x,y\in\mathcal{O}_{F}}\vert h_g(x,y)\rvert$ in this case.

In summary, \eqref{eq:h-h',last case} is less than or equal $r_{1}\lvert \gamma\rvert^{n}$ for some $r_{1} > 0$ which does not depend on $a,b,c,n$.
Hence it converges to zero uniformly.
\end{proof}

\subsection{Proof of Proposition~\ref{prop:final case}}
\label{sec:empty2}
Let $\pi\subset (\Ind_{B}^{G}\chi)^{\sm}$ be a non-zero subrepresentation.
We have a non-zero map $\pi\hookrightarrow (\Ind_{B}^{G}\chi)^{\sm}$.
Hence by Frobenius reciprocity, we have a non-zero $T$-equivariant map $\pi_{U}\to \chi$.
The composition $\pi_{N}\to\pi_{U}\to \chi$ is non-zero $Z_{L}$-equivariant map.
Therefore $\pi_{N}^{Z_{L} = \chi}$ is non-zero and it is a subrepresentation of $((\Ind_{B}^{G}\chi)^{\sm})_N^{Z_L = \chi} \cong (\Ind_{B\cap L}^{L}\chi)^{\sm}$ (cf.\ subsection \ref{sec:jacquet-modules}).

By our assumption we have $\chi_{1}\ne \chi_{2}$.
Therefore, $(\Ind_{B\cap L}^{L}\chi)^{\sm}$ is irreducible or has as socle a twist of the Steinberg representation.
We define functions $f_n$ ($n \ge 1$) as at the beginning of subsection~\ref{sec:density-arg} and recall that they depend on a smooth function $g : F \to C$ vanishing outside $\mathcal O_F$ and a constant $\gamma \in C^\times$.
The choice of $g$ and $\gamma$ will come from Lemma~\ref{lem:f_n span dense subspace}.

If $(\Ind_{B\cap L}^{L}\chi)^{\sm}$ is irreducible, then $\pi_{N}^{Z_{L} = \chi} = (\Ind_{B\cap L}^{L}\chi)^{\sm}$, or equivalently $$\pi^{N\cap K,Z_{L}^{+} = \chi}=(\Ind_{B}^{G}\chi)^{\sm,N\cap K,Z_{L}^{+} = \chi}.$$
Then $f_{n} \in \pi$ for all $n \ge 1$ and so $\pi$ is dense in $(\Ind_{B}^{G}\chi)^{\cts}$ by Lemma~\ref{lem:f_n span dense subspace}.
Hence $(\Ind_{B}^{G}\sigma)^{\cts}$ is irreducible by Theorem~\ref{thm:criterion}.

Assume that $(\Ind_{B\cap L}^{L}\chi)^{\sm}$ is reducible with socle a twist of the Steinberg representation.
Hence $\pi^{N\cap K,Z_{L}^{+} = \chi} \cong \pi_{N}^{Z_{L} = \chi}$ also contains a twist of the Steinberg representation.
The socle of $(\Ind_{B\cap L}^{L}\chi)^{\sm}$ is the kernel of a surjective morphism $\varphi\colon (\Ind_{B\cap L}^{L}\chi)^{\sm}\onto (\chi_{1}\lvert \cdot\rvert_F^{-1}\circ \det)\boxtimes\chi_{3}$.
If the support of $f\in (\Ind_{B\cap L}^{L}\chi)^{\sm}$ is contained in $(U\cap L)\dot{s}_{1}(B\cap L)$, then we can normalize $\varphi$ so that $\varphi(f) = \int_{U\cap L}f(x\dot{s}_{1})dx$.
Hence, assuming that $\int_{\mathcal{O}_{F}}g(x)dx = 0$, we have
\begin{align*}
\varphi(f_{n}|_{L}) & = \int_{U\cap L}f_{n}(x\dot s_{1})dx\\
& = \int_{F}f_{n}\left(\begin{pmatrix}
1 & x & 0\\
0 & 1 & 0\\
0 & 0 & 1
\end{pmatrix}\dot s_{1}\right)dx\\
& = \gamma^n \int_{(\varpi_{F}^{n})}g\left(\frac{x}{\varpi_{F}^n}\right)dx\\
& = \gamma^n \lvert\varpi_{F}^{n}\rvert_F\int_{\mathcal{O}_{F}}g(x)dx = 0.
\end{align*}
Therefore, in this case, $f_n|_L$ lies in the socle of $(\Ind_{B\cap L}^{L}\chi)^{\sm}$, i.e.\ $f_n \in \pi^{N\cap K,Z_{L}^{+} = \chi} \subset \pi$.
By Lemma~\ref{lem:f_n span dense subspace} $\pi$ is dense in $(\Ind_{B}^{G}\chi)^{\cts}$ and hence $(\Ind_{B}^{G}\chi)^{\cts}$ is irreducible by Theorem~\ref{thm:criterion}.

\newcommand{\etalchar}[1]{$^{#1}$}
\providecommand{\bysame}{\leavevmode\hbox to3em{\hrulefill}\thinspace}
\providecommand{\MR}{\relax\ifhmode\unskip\space\fi MR }
% \MRhref is called by the amsart/book/proc definition of \MR.
\providecommand{\MRhref}[2]{%
  \href{http://www.ams.org/mathscinet-getitem?mr=#1}{#2}
}
\providecommand{\href}[2]{#2}

\end{document}